\documentclass[11point]{amsart}
\usepackage{amsmath, xcolor, datetime}
\usepackage{amscd}

\usepackage{amssymb,amsmath,amscd,epsfig,amsthm,enumerate,import,graphicx,pinlabel}
\usepackage{caption}
\usepackage{subcaption}

\usepackage{geometry}                
\newtheorem{theorem}{Theorem}[section]
\newtheorem{proposition}[theorem]{Proposition}
\newtheorem{lemma}[theorem]{Lemma}
\newtheorem{corollary}[theorem]{Corollary}
\theoremstyle{definition}
\newtheorem{definition}[theorem]{Definition}

\theoremstyle{remark}
\newtheorem{remark}[theorem]{Remark}
\newtheorem{remarks}[theorem]{Remarks}

\newcommand{\be}{\begin{equation}}
\newcommand{\ee}{\end{equation}}

\makeatletter
\newcommand{\tpitchfork}{%
  \vbox{
    \baselineskip\z@skip
    \lineskip-.52ex
    \lineskiplimit\maxdimen
    \m@th
    \ialign{##\crcr\hidewidth\smash{$-$}\hidewidth\crcr$\pitchfork$\crcr}
  }%
}
\makeatother

\newcommand{\R}{{\bf R}}

\newcommand{\w}{{\bf w}}

\newcommand{\wf}{\mbox{WF}}

\newcommand{\bbC}{{\mathbb C}}

\newcommand{\bbQ}{{\mathbb Q}}

\newcommand{\bbN}{{\mathbb N}}
\newcommand{\bbR}{{\mathbb R}}

\renewcommand{\r}{\overline{r}}
\newcommand{\x}{\overline{x}}
\newcommand{\y}{\overline{y}}
\renewcommand{\w}{\overline{w}}

\newcommand{\calD}{{\mathcal D}}
\newcommand{\calB}{{\mathcal B}}

\newcommand{\calQ}{{\mathcal Q}}

\newcommand{\calK}{{\mathcal K}}

\newcommand{\calM}{{\mathcal M}}

\newcommand{\calF}{{\mathcal F}}
\newcommand{\calC}{{\mathcal C}}
\newcommand{\calR}{{\mathcal R}}
\newcommand{\calS}{{\mathcal S}}

\newcommand{\calT}{{\mathcal T}}

\newcommand{\frakY}{{\mathfrak Y}}

\newcommand{\tr}{{\mbox{Tr}}}

\newcommand{\rest}{\upharpoonleft}

\renewcommand{\Box}{\square}

\reversemarginpar

  \DeclareSymbolFont{bbold}{U}{bbold}{m}{n}
\DeclareSymbolFontAlphabet{\mathbbold}{bbold}

  \newcommand{\prj}{\pi}
  \newcommand {\Oph}{\operatorname{Op}_h}
  \newcommand {\tpsi}{t_{\psi}}
  \newcommand{\pop}{\mathbb{P}}
  \newcommand{\popapp}{\pop_{\operatorname{appr}}}
  \newcommand{\gef}{F}

\newcommand{\tX}{\tilde{X}}

\newcommand{\spec}{\operatorname{spec}}
\newcommand{\mch}{\mathcal{H}}
\newcommand{\mchy}{\mathcal{H}_Y}
\newcommand{\supp} {\operatorname{supp}}

  \newcommand{\Xinf} {X_\infty}
  \newcommand{\Xtil}{\tilde{X}}
  \newcommand{\DY}{\Delta_Y}
  \newcommand{\DX}{\Delta_X}
\newcommand{\DXt}{\Delta_{\Xtil}}
\newcommand{\psp}{\psi_{sp}}
\newcommand{\sro}{(I-h^2\DY)^{1/2}}
  \DeclareSymbolFontAlphabet{\mathbbold}{bbold}
  \newcommand{\bbI}{\mathbbold{1}_{[0,1]}}
  \newcommand{\tS}{{S}}
\newcommand{\Ss}{S^{\#}}

\newcommand{\Ph}{P}

\newcommand{\Sphere}{\mathbb{S}}
\newcommand{\mcB}{\mathcal{B}}
\newcommand{\SU}{S_U}
\newcommand{\vol}{\operatorname{vol}}
\newcommand{\Z}{{\bf Z}}
\newcommand{\Natural}{\bf{N}}
\newcommand{\ts}{\tilde{s}}
 
 \newcommand{\C}{\bbC}
 \newcommand{\N}{\bbN}
  
\begin{document}

\title[Semiclassical structure of the scattering matrix]{The semiclassical structure of the scattering matrix for a manifold with infinite cylindrical end}

\author{T. J. Christiansen}
\address{Department of Mathematics, University of Missouri, Columbia, MO, USA}
\email{christiansent@missouri.edu}
\author{A. Uribe}
\address{Mathematics Department\\
University of Michigan\\Ann Arbor, MI, USA}
\email{uribe@umich.edu}

\begin{abstract}

We study the microlocal properties of the scattering
matrix associated to the semiclassical 
Schr\"odinger operator $P=h^2\Delta_X+V$ on a Riemannian
manifold with an infinite cylindrical end.
The scattering matrix at $E=1$
is a linear operator 
$S=\tS_h$ defined on a Hilbert subspace of $L^2(Y)$ that parameterizes the continuous spectrum of $P$ at energy $1$.  
Here $Y$ is the cross section of the end of $X$, which is not necessarily connected.
We show that, under certain assumptions, microlocally
$S$ is a Fourier integral operator associated to the graph of the scattering map $\kappa:\mathcal{D}_{\kappa}\to T^*Y$, with $\mathcal{D}_\kappa\subset T^*Y$.
The scattering map $\kappa$ and its domain $\calD_\kappa$ are 
determined by the Hamilton flow of the principal symbol of $P$.
As an application we prove that, under additional hypotheses on the scattering map,
the eigenvalues of the associated unitary scattering matrix
are equidistributed on the unit circle.  

\end{abstract}

\maketitle
\date{}


\section{Introduction} 

For certain Euclidean or asymptotically conic scattering problems it is known that the scattering matrix quantizes the scattering relation, a mapping determined by the bicharacteristic flow of the principal symbol of the operator in question, e.g. 
\cite{Ale05,Ale06,ABR,Ing, HW}.  
Here we consider this problem for a  class of manifolds with infinite cylindrical ends with an application to the 
equidistribution of phase shifts of the unitary scattering matrix.   Our results are related to results of  \cite{ZZ}, but are quite different in methodology and technically 
apply to different classes of manifolds.

Throughout this paper, $(X,g)$ will denote a smooth  connected 
Riemannanian manifold with {\em infinite cylindrical end}.   That is, $X$ has a decomposition as
$X=X_C\cup\Xinf$, where $X_C$ is a smooth compact manifold  with boundary $\partial X_C=Y$,
and $X_\infty\cong (-4,\infty)\times Y$.  More precisely, if we denote by $g_Y$
the restriction of		 $g$ to $TY=T\partial X_C$ (this
is a metric on $Y$), 
we assume that $\Xinf$ is isometric to  $(-4,\infty)\times Y$ 
with the product metric $(dr)^2+g_Y$ where $r$ is the natural coordinate on $(-4,\infty)$.
We do not necessarily assume that $Y$ is connected.   
For convenience, we extend $r$ to a smooth function on $X$, so that $r\leq -4$ on $X_C$. 
The (non-negative) Laplacians on $X$ and $Y$ are denoted by $\DX$, $\DY$ respectively.  

The purpose of this paper is to study the microlocal properties of the scattering
matrix associated to the semiclassical 
Schr\"odinger operator 
\[
P = h^2\Delta_X + V.
\]  
Here $V=V(h,x)=V_0(x)+h^2V_2(x)$, with $V_0,\; V_2\in C_c^\infty(X_C)$.
The scattering matrix 
is a linear operator 
$S=S(h):\mathbbold{1}_{[0,1]}(h^2\DY)L^2(Y)\rightarrow \mathbbold{1}_{[0,1]}(h^2\DY)L^2(Y)$
(where $\mathbbold{1}_{I}$ denotes the characteristic function of the interval $I$), whose definition
we recall in Section  \ref{ss:sm}.  The space 
\begin{equation}\label{eq:HY}
\mchy:= \bbI(h^2\DY)L^2(Y)=\{f\in L^2(Y) \mid \bbI(h^2 \DY)f=f\}
\end{equation}
parameterizes the continuous spectrum of $P$ at energy $1$.

We fix once and for all a semiclassical quantization scheme denoted $\Oph$, associating to compactly
supported smooth functions $\psi$ on $T^*Y$ semiclassical pseudodifferential operators $\Oph(\psi)$ on $L^2(Y)$.  
Our main Theorems,  \ref{thm:v1} and \ref{thm:v2}, state that, under certain assumptions on the
resolvent  $(\Ph-1-i0)^{-1}$,
for suitable functions $\psi\in C_c^\infty(T^*Y)$ the compositon 
$S\circ\Oph(\psi)$ is a Fourier integral operator associated to the graph of the scattering map $\kappa$.  We define $\kappa$ in Section \ref{ss:scattmap}.
Under some additional hypotheses, including that the set of fixed points of
$\kappa^m$ has measure zero for all $m=1,2\ldots$, 
we use these results to prove in Theorem \ref{thm:equidist} that the eigenvalues of the associated unitary scattering matrix $S_U$ (unitary on $\mch_Y$),
are equidistributed on $\Sphere^1$.

The main results are precisely stated in Section \ref{section:MainRs}.

\subsection{The scattering map}\label{ss:scattmap}
The  {\em scattering map} $\kappa$ is defined on an open subset $\calD_\kappa$
of the open unit tangent ball bundle  of $Y$,
$$\calB = \{(y,\eta)\in T^*Y \mid  |\eta|<1\}.$$
The map $\kappa$ is analogous to the scattering map of \cite{Ing},
and related to the scattering relation of \cite{Ale05, Ale06} and others. 

The definition involves
the Hamilton flow $\Phi_t$ of $p(x,\xi) = |\xi |^2 + V_0(x)$, the principal symbol of $P$, on $T^*X$.  Note that since 
$p_{\upharpoonright T^* \Xinf}=|\xi|^2$, the projections of the trajectories of $\Phi_t$ 
in $T^*\Xinf\subset T^*X$ to $\Xinf$ are geodesics on the 
product manifold $(-4,\infty)\times Y$.  

\begin{definition}\label{def:smap} 
A point $(y_-,\eta_-)\in\calB$ is in the domain $\calD_\kappa$ of the scattering map $\kappa$ 
if and only if  the
trajectory of $(0,y_-,-\sqrt{1-|\eta_-|^2},\eta_-)\in T^*\Xinf\subset T^*X$ under
the Hamilton flow $\Phi_t$ of $p$ is not forward trapped, that is, if 
and only if 
\[
\exists \; T>0 \;\text{so that}\; \forall \;t>T \qquad \Phi_t(  0,y_-,-\sqrt{1-|\eta_-|^2},\eta_-)\in X_\infty.
\]
 For
such $(y_-, \eta_-)$, there is a $t_+=t_+(y_-,\eta_-)>0$ and a 
$(y_+,\eta_+)=(y_+(y_-,\eta_-), \eta_+(y_-,\eta_-))\in T^*Y$ such that 
\[
\Phi_{t_+}(  0,y_-,-\sqrt{1-|\eta_-|^2},\eta_-)= (0,y_+,\sqrt{1-|\eta_+|^2}, \eta_+),
\]
and we define
\[\kappa(y_-,\eta_-):=(y_+,\eta_+).
\]
Thus $\kappa: \calD_\kappa\to\calB\subset T^*Y$.
\end{definition}
\begin{remarks}
Some remarks may be in order.  
	\begin{enumerate}
\item Under the hypotheses of the definition, 
let 
$(x(t),\xi(t))=\Phi_t(  0,y_-,-\sqrt{1-|\eta_-|^2},\eta_-)$.
Since $x(0)=(r(0),y(0))=(0,y_-)\in \Xinf$ and
$\dot{r}(0)<0$, $x(t)\in X_C$ for some 
$t>0$.  
The non-trapping condition means that at some later time the trajectory $(x(t),\xi(t))$ will
exit $T^*X_C$ and lie over $X_\infty$.


\item If $V\equiv 0$, the map $\kappa$ is the billiard map of $\{r\leq 0\}$,  a Riemannian manifold with boundary.

\item The scattering map does depend on the choice of 
decomposition of $X$ as $X=X_C\cup \Xinf$, since this choice determines
the location of the set $\{r=0\}\subset \Xinf$.  
We will see in Remark \ref{rmk:changeInKappa}
	that a different choice of origin for the $r$ coordinate
results in a scattering map $\kappa':\calD_{\kappa '}\to\calB$ which is of the form
\begin{equation}\label{eq:changingOrigin}
	\kappa' = \vartheta\circ\kappa\circ\vartheta \quad\text{and}\quad \calD_{\kappa '} = \vartheta^{-1}(\calD_\kappa),
	\end{equation}
for a certain canonical transformation $\vartheta: \calB\to\calB$.  (Note that $\kappa$ and $\kappa'$ are {\em not} conjugate.)

\item  Introduce the notation for all $ \overline{y}= (y,\eta)\in T^*Y,\ \overline{y}'=(y,-\eta)$.
Then, using the time-reversibility of the flow $\Phi_t$, it is
not hard to see that $\kappa ( \kappa (\overline{y})')=\overline{y}'$.
Therefore $\kappa$ is one-to-one.

\item Examples show  that $\calD_\kappa$
can be a proper subset of $\calB$. 

\end{enumerate}
\end{remarks}

\subsection{The scattering matrix}\label{ss:sm}

For a manifold with an infinite cylindrical end, the scattering matrix  for the operator $P=h^2\DX+V$ is a linear operator from $\mchy$ to itself, where $\mchy$ is 
defined in (\ref{eq:HY}).
 Thus the scattering matrix acts on a finite-dimensional space whose dimension 
increases as $h>0$ decreases, and thus can in fact be identified with 
a matrix, albeit one whose dimension changes with $h$.  In \cite{tapsit, chr95,par} the scattering matrix is defined via its entries in a 
particular basis.  It is more convenient here to take an approach like
 that is used in 
 the Euclidean or cylindrical end case in \cite[Sections 2.7, 7.3]{lrb}, defining the scattering matrix by its action on any element of $\mchy$.  That the two approaches yield the same 
 operator is well-known, easy to check, and is a consequence of our  proof of Lemma  \ref{l:tSwelldefined}.
 
 We also note that
 there are several conventions in the literature as to exactly which operator is referred to as the scattering matrix.  One, which 
we shall denote $S_U$, is normalized to be unitary on $\mchy$; this
is found in \cite{chr95, par}, for example.   We shall work primarily with the unnormalized scattering matrix that we denote $S$, found in \cite{tapsit}.  The two are related by 
$S_U=(I-h^2\DY)_+^{1/4}S(I-h^2\DY)_+^{-1/4}$, 
where $(\bullet)_+$ is the Heaviside function.
We shall refer to $S_U$ as the unitary scattering matrix.

Let $\sro$ be the operator on $L^2(Y)$ defined by the spectral theorem, with
non-negative real and imaginary parts. 
Suppose $\gef \in \langle r \rangle ^{1/2+\epsilon}H^2(X)$ for all $\epsilon>0$ and $\gef$ is in the null space of $\Ph-1$.  
 Suppose in
addition that $h^2$ is not the reciprocal of an eigenvalue of $\DY$.
Then on $\Xinf$ a separation of variables argument shows that we can write
\begin{equation}
\label{eq:uexp}
\gef\rest_{\Xinf}(r,y)=e^{-ir\sro/h}\bbI(h^2\DY)f_{-}+ e^{ir\sro/h}f_{+}
\end{equation}
for some functions $f_{\mp}\in L^2(Y)$.
We shall refer to $f_{-}$ as the incoming data, and $f_{+}$ as the outgoing data.
If  $1\not \in \spec(h^2\DY)$, then the (unnormalized) scattering matrix $S=S(h)$ is such that:
\begin{equation}\label{eq:tSdef}
S\left( \bbI(h^2\DY) f_{-}\right) := \bbI(h^2\DY)f_{+}.
\end{equation}

More precisely:
\begin{lemma}\label{l:tSwelldefined}
If $1\not \in \spec(h^2\DY)$, for every $f\in\mchy$ there exists $\gef\in \langle r \rangle ^{1/2+\epsilon}H^2(X)$  in the null space of $\Ph-1$
such that  (\ref{eq:uexp}) holds with $\bbI(h^2\DY) f_{-}=f$, and the relation
$S(f)= \bbI(h^2\DY)f_{+}$ defines an operator $S:\mchy\rightarrow \mchy$.
Moreover, if $1\in \spec(h_0^2\DY)$, then
 $\lim_{h'\uparrow h_0}S(h')$ exists as a bounded operator.
\end{lemma}
If $1\in \spec(h_0^2\DY)$, then we define
$S(h_0)=\lim_{h'\uparrow h_0}S(h')$.   

Although the results of Lemma \ref{l:tSwelldefined} are known (e.g. \cite{tapsit, chr95,par, lrb}), for the convenience of the reader we give a proof  in 
Section \ref{s:Posm}.  
Additionally, the proof shows the operator $S$ is (up to sign conventions) consistent with the 
non-unitary scattering matrices of \cite{tapsit,chr95,par}.

Like the scattering map, the scattering matrix depends on the choice of 
coordinate $r$ on the end, which corresponds to fixing the decomposition
$X=X_C\cup \Xinf$.  For example, if for $c_0>-4$ we instead 
write $X=X_C'\cup \Xinf'$, with $X_C'=X_C\cup\{x=(r,y)\in \Xinf \mid  -4<r\leq c_0\}$
and $\Xinf'=\Xinf \setminus \{x=(r,y)\in \Xinf \mid  -4<r\leq c_0\}$, then the coordinate in the new decomposition is $r'=r-c_0-4$. 
With $S'$ denoting the scattering matrix for the decomposition
$X_C'\cup \Xinf'$, $S'= e^{i(c_0+4)(I-h^2\DY)^{1/2}_+/h} S e^{i(c_0+4)(I-h^2\DY)^{1/2}_+/h}$.
Compare this with the corresponding change in the scattering map, (\ref{eq:changingOrigin}).


\subsection{ Main results} \label{section:MainRs}
In  our main theorem we assume that an appropriate
cut-off resolvent is bounded at high energy--this is hypothesis (\ref{eq:rbdhyp}) of
Theorem \ref{thm:v1}.  Section \ref{s:examples} contains examples
of manifolds and potentials for which 
this hypothesis holds,
and \cite[Theorem 3.1]{CDI} gives a technique for constructing
such
 manifolds.  Section \ref{s:examples} also contains examples for which the weaker resolvent bound (\ref{eq:rbdhyp2}) and the other
hypotheses of Theorem \ref{thm:v2} hold.

Throughout the paper, we use the notation $(P-1\pm i0)^{-1}= \lim_{\delta\downarrow 0} (P-1\pm i \delta)^{-1}$.

\begin{theorem}\label{thm:v1}
Suppose there are constants $C_0,\;N_0,\;h_0>0$ so that
\begin{equation}\label{eq:rbdhyp}
 \|\mathbbold{1}_{[0,1]}(h^2\DY)\mathbbold{1}_{[0,1]}(r)(\Ph-1-i0)^{-1}\mathbbold{1}_{(-\infty,0]}(r)\|\leq C_0h^{-N_0}\; \text{for $0<h\leq h_0$}.
\end{equation}  Let $\psi\in C_c^\infty(T^*Y)$ have its support in the domain of the scattering map.
Then  for $0<h<h_0$ $S\Oph(\psi) $ and $\SU\Oph(\psi)$ are
semi-classical Fourier integral operators associated with the graph of the scattering map $\kappa$.
\end{theorem}
Proposition \ref{p:smident}
gives a more explicit expression for the scattering matrix using the 
Schr\"odinger propagator and some operators which map
between $L^2(Y)$ and  $L^2(\Xinf)$.
This explicit expression shows how the scattering matrix is a quantum analog
of the scattering map defined in Section \ref{ss:scattmap}; see also Section \ref{s:idea}.

We remark here that there is some flexibility in choosing the exact cut-offs in (\ref{eq:rbdhyp}): we could replace  $\mathbbold{1}_{[0,1]}(r)$ by $\mathbbold{1}_{[b,c]}(r)$
and $\mathbbold{1}_{(-\infty,0]}(r)$ by $\mathbbold{1}_{(-\infty,a]}(r)$ if 
$-4<a<b<c<\infty$.  Although we do not prove this, Section \ref{s:reX} proves some results in this direction.

A more restrictive assumption on the manifold and operator than in Theorem \ref{thm:v1} allows us to 
make a weaker assumption on the resolvent bound.   In this next theorem we 
assume that $X$ is diffeomorphic to $\R\times Y_0$, but we do not assume
that the metric is globally a product metric. In Section \ref{s:examples}
 we give
 two families of examples for which the metrics on $X$ have a warped product structure
and the resolvent for $P=h^2\DX$ satisfies
 the estimate (\ref{eq:rbdhyp2}), but which
have quite different trapping  properties and quite different quantitative behavior of the eigenvalues of $\DX$.
\begin{theorem} \label{thm:v2} Let $(Y_0,g_{Y_0})$ be a smooth compact Riemannian manifold,
and let $g$ be a metric on $X=\R\times Y_0$ which is the product
metric $(dr)^2+g_{Y_0}$ outside of a compact set. Let $P=h^2\DX+V$ satisfy $[P,\Delta_{Y_0}]=0$.
Suppose for any $\epsilon>0$ there are constants $C_0=C_0(\epsilon),\;N_0=N_0(\epsilon),\;h_0=h_0(\epsilon)>0$ so that
\begin{equation}\label{eq:rbdhyp2}
 \|\mathbbold{1}_{[0,1-\epsilon]}(h^2\DY)\mathbbold{1}_{[0,1]}(r)(\Ph-1-i0)^{-1}
\mathbbold{1}_{(-\infty,0]}(r)\|\leq C_0h^{-N_0}\; \text{for $0<h\leq h_0$}.
    \end{equation}
Let $\psi\in C_c^\infty(T^*Y)$ have its support in the domain of the scattering map.
Then  for $0<h<h_0$ $S\Oph(\psi) $ and $\SU\Oph(\psi)$ are
 semi-classical Fourier integral operators associated with the graph of the scattering map $\kappa$. 
\end{theorem}

These two theorems are proved by combining the results of Propositions \ref{p:mainML} and \ref{p:smident}.

In Section \ref{s:eps} we use these theorems to prove Theorem \ref{thm:equidist}.  This
shows that  under some additional hypotheses  in the semiclassical limit the eigenvalues of the 
unitary scattering matrix  $S_U$ are equidistributed.

We now comment on the resolvent estimates, (\ref{eq:rbdhyp}) and (\ref{eq:rbdhyp2}).
In Euclidean or hyperbolic scattering settings
 bounds on a cut-off resolvent of a 
semiclassical operator are well known under non-trapping assumptions on the
bicharacteristic flow of the associated Hamiltonian.   Moreover, 
some estimates are known under assumptions that the trapping is relatively mild; see, for example, \cite[Section 3]{Zwo} for a recent survey.  
All of the operators we consider here have nontrivial
trapping, as each geodesic in $Y$ corresponds to a trapped bicharacteristic
of $P$ in $\{p=1\}\cap\{r=c\}\subset T^*\Xinf$ for any $c>-4$. 

 For 
manifolds with infinite cylindrical ends, $(\Ph-1-i0)^{-1}=(h^2\DX+V-1-i0)^{-1}$
can have poles for a sequence of $h_j\downarrow 0$.
  For example, 
let $(Y_0,g_0)$ be a compact Riemannian manifold, and 
consider 
the simplest case $X=\R\times Y_0$ with the product metric.  Then
for any nontrivial $\chi \in C_c^\infty(X)$, $\chi (h^2\DX-1-i0)^{-1}\chi$
has a pole whenever $1/h^2$ is an eigenvalue of $\Delta_{Y_0}$--
though in this case  including a spectral projection in $\DY$ as is done
in (\ref{eq:rbdhyp}), as well as a spatial cut-off,
 is enough to ensure a bound which is polynomial in $h$.
Theorem 3.1 of \cite{CDI} gives a technique of constructing manifolds $(X,g)$
and operators $P=h^2\DX+V$ so that for any $\chi \in C_c^\infty(X)$, 
$\| \chi(P-1-i0)^{-1}\chi\|$ is polynomially bounded in $h$.  In 
Section \ref{s:examples} below we give some examples, most using results
from \cite{CDI}, for which (\ref{eq:rbdhyp}) or (\ref{eq:rbdhyp2}) holds.

 In an effort to simplify the exposition, our results are for the scattering matrix at fixed energy $1$, with corresponding hypotheses (\ref{eq:rbdhyp}) and (\ref{eq:rbdhyp2}) 
on the resolvent at energy $1$.  However, as is well known a rescaling can be used to prove corresponding results at other positive energies.  Let $E>0$,
and write $P-E=E(\frac{1}{E}P-1)= E(\frac{1}{E}(\DX+V)-1)$.  Setting $(X',g')=(X, Eg)$, we have $\Delta_{X'}=\frac{1}{E}\DX$.  By defining $r'=E^{-1/2}(r+4) -4$, we see that we can decompose $X'=X_C'\cup X'_\infty$ so that $g'\upharpoonleft_{X'_\infty}= (dr')^2+Eg_Y$, as required in our definition of a manifold with infinite cylindrical end.  
Then  results for the scattering matrix of $P'=\Delta_{X'}+\frac{1}{E}V$ at energy $1$ then imply results for the scattering matrix of $P$ at energy $E$.

\subsection{Idea of the proof}  \label{s:idea}

In order to prove the theorem, we construct the {\em Poisson operator} $\pop$, or,  more precisely, the 
Poisson operator multiplied on the right by $\Oph(\psi)$, $\pop\Oph(\psi)$.    We define the Poisson operator below, and show in Section \ref{s:Posm} that it is in fact well-defined.  

\begin{definition} \label{d:pop} Suppose $1\not \in \spec(h^2\DY)$.  The Poisson operator is a linear operator 
$\pop:L^2(Y)\rightarrow \langle r\rangle ^{1/2+\delta} H^2(X)$ for any $\delta>0$
so that for $f\in L^2(Y)$, 
$(\Ph-1)\pop f=0$ and $\pop f$ has specified incoming data:
\begin{equation}
\label{eq:poponend}
(\pop f)\upharpoonleft_{ \Xinf}= e^{-ir \sro /h} \mathbbold{1}_{[0,1]}(h^2\DY)f+ e^{ir \sro /h}f_+
\end{equation}
for some $f_+\in L^2(Y)$.  Moreover,  we require that $\langle \pop f, g\rangle=0$ for any $L^2$ eigenfunction $g$  of $P$ with eigenvalue $1$.
\end{definition}
We note 
a separation of variables on the end $\Xinf$ shows that any 
$L^2$ eigenfunction of $P$ must be exponentially decreasing on $\Xinf$, so that its product with an element of $\langle r \rangle^{1/2+\delta} L^2(X)$ is integrable.
Thus the pairing $\langle \pop f, g \rangle$ which we take to mean  $\langle \pop f, g \rangle=\langle  \langle r\rangle ^{-1/2+\delta} f, \langle r \rangle^{1/2+\delta} g \rangle$ makes sense.
Without the restriction involving the eigenfunctions with eigenvalue $1$, $\pop$ is not uniquely determined at values of $h$ for which $1$ is an eigenvalue of $P=P(h)$. 

By the definition of the scattering matrix,
$\tS \mathbbold{1}_{[0,1]}(h^2\DY)f =\mathbbold{1}_{[0,1]}(h^2\DY)f_+$, where $f_+$ is as in (\ref{eq:poponend}).

We will now outline the ideas behind the microlocal construction of $\pop\Oph(\psi)$ below, omitting details here for clarity.

In the construction of our initial approximation of $\pop \Oph(\psi)$ we shall use cut-off functions to
 piece together three terms: on the end $\Xinf$ we use both the incoming and the outgoing resolvents on 
the product $\tX=\R \times Y$, and on a compact 
subset of $X$ we use (roughly) $\int_0^{\tpsi} e^{it/h}e^{-it\Ph/h} dt$.   The time $\tpsi$ is chosen to ensure the bicharacteristics  of the Hamiltonian flow that 
start at points $(0,y, -\sqrt{1-|\eta|^2},\eta)\in T^*\Xinf\subset T^*X$  with $(y,\eta)$ in the support of $\psi$ have returned to the portion of $T^*X$ with $r\geq 0$ (thus lying in $T^*\Xinf)$
 by time $\tpsi$.  In other words, $t_{\psi}\geq \sup_{(y,\eta)\in \supp \psi}t_{+}(y,\eta)$, where $t_+$ is the function defined in Section \ref{ss:scattmap}.

In studying the outgoing and incoming resolvents on the product manifold $\tX=\R\times Y$, the operators
\begin{equation}\label{natMappings}
T _{\pm} f = \int_{\R}
e^{\mp ir' (\sro)/h}f(r',\bullet)dr',\; T_{\pm}: L^2_c(\R\times Y)\rightarrow L^2(Y)
\end{equation}
arise naturally.   In order to ensure our approximation to 
$\pop\Oph(\psi)$ has the desired incoming data, we shall need a right inverse of $T_-$.  Let $\chi\in C_c^\infty
((-1/4,0);\R_+)$ satisfy $\int \chi(r)dr=1$, and set, for $g\in C^\infty(Y)$,
\begin{equation}\label{rInverse}
R_{\pm}g= \chi(r )e^{\pm ir(I-h^2\DY)^{1/2}_+/h}g,
\end{equation}
so that $T_\pm R_\pm=I$.
We shall apply the incoming resolvent on $\R\times Y$ to $\sro R_-\Oph(\psi)$, and then the outgoing resolvent to an operator determined by $e^{-i\tpsi \Ph/h}\sro R_-\Oph(\psi)$.
Our assumptions on the resolvent in Theorem \ref{thm:v1} or \ref{thm:v2} ensure that the approximation to the Poisson operator which we construct is in fact close to the genuine one.  Proposition
\ref{p:smident} gives, up to small error, an explicit expression for the (cut-off) scattering matrix involving $R_-$, $e^{-i\tpsi \Ph/h}$, and $T_+$.  In a rough sense,
the resulting expression for the scattering matrix parallels the construction of the scattering map $\kappa$.

Our construction thus involves three semiclassical Fourier integral operators: $R_-\psp(h^2\DY)$, $e^{-i\tpsi \Ph/h}$, and $T_+\psp(h^2\DY)$.  Here 
$\psp\in C_c^\infty([0,1))$ is chosen to be $1$ on a sufficiently large set.    Each of these is a well-studied operator in its own right (though 
in $R_-$ and $T_+$, an ``$r$" occurs where we might more naturally expect to find ``$t$").  Part of the proof of the theorems is to carefully check the compositions which occur, not only in the 
expression for the scattering matrix, but also elsewhere in the construction of $\pop$.  This is done in Section \ref{s:mproperties}.  
The constructions of $\pop \Oph(\psi)$ and  $\tS \Oph(\psi)$ are carried out in Section 
\ref{s:gef}.

\subsection{Background and related work}   An introduction to the
 spectral and scattering theory of manifolds with infinite 
cylindrical ends can be found in \cite{Gol,GuiL,tapsit}, with further 
results for the scattering matrix in \cite{chr95, par}.   A relatively short self-contained introduction may also be found in \cite[Section 2]{CD2}.
The papers \cite{chrst,ChTa} use a detailed microlocal analysis of the scattering matrix applied to a specific function in an inverse problem.

In \cite{ZZ} Zelditch and Zworski consider a family of surfaces of revolution 
having a single connected asymptotically cylindrical end,  proving a result for the pair correlation measure of the phase shifts of the (unitary) scattering matrix.
This is stronger than our equidistribution result, Theorem \ref{thm:equidist}, but is for a particular class of surfaces of revolution.  
Additionally, Proposition
3 of \cite{ZZ} shows 
that for the surfaces under consideration the truncated scattering 
matrix (in their setting, $\tS \mathbbold{1}_{[\epsilon, 1-\epsilon]}(h^2\DY)$)
is a semiclassical quantum map associated to the scattering map $\kappa$.  The paper \cite{ZZ} uses the warped product structure of the surface 
and a separation of variables argument to reduce the problem to a study of a family of one-dimensional problems.

There are many papers which use microlocal analysis
to study the properties of the scattering matrix in 
Euclidean scattering.   Alexandrova \cite{Ale05,Ale06} shows
that under suitable assumptions the scattering amplitude for a 
compactly supported perturbation of the semiclassical Euclidean
Laplacian quantizes the scattering relation (see also \cite{Ale18, ABR}). 
A related result for the asymptotically conic setting is
\cite{HW}.  Ingremeau \cite{Ing} studies 
mapping properties of the scattering matrix
on Gaussian coherent states for (non-trapping) 
semiclassical Schr\"odinger operators on $\R^n$.     Our proofs of Theorems  \ref{thm:v1} and \ref{thm:v2} have been 
influenced by both \cite{Ale05} and \cite{Ing}, as well as by 
\cite[Section 3.11]{DyZw}. 
There are many other results which use microlocal techniques to find
asymptotics of the scattering matrix in Euclidean settings.  We mention
just a few, \cite{Vai,Gui,RT,Mic} and refer the reader to the cited papers for
further references. 

  The distribution of phase shifts has been studied in a number of Euclidean settings, e.g. \cite{BP,DGHH, GHZ,Ing2, GI,GH}.  Some of these papers
  use the results of 
Alexandrova or Ingremeau on the microlocal structure of the scattering matrix.   Our Theorem \ref{thm:equidist} is an application of Theorems \ref{thm:v1} or \ref{thm:v2} to prove an 
equidistribution result in the cylindrical end setting.   

\vspace{2mm}
\noindent
{\bf Acknowledgements.} The authors thank Kiril Datchev and Maciej Zworski for helpful
conversations and suggestions.  In addition, the authors thank K. Datchev for making the first versions of some of the figures used in this paper.   
The first author  gratefully acknowledges
the partial support of an M.U. Research Leave
and a Simons Foundation collaboration grant. Moreover, this material is based
in part upon work supported by the National Science Foundation under Grant No. 1440140, while the authors were in residence at the Mathematical Sciences 
Research Institute in Berkeley, California, during the fall 2019 semester.

\section{Examples for which one of
 the resolvent estimates holds}\label{s:examples}
In this section we give some examples of manifolds for which 
the estimates (\ref{eq:rbdhyp}) or (\ref{eq:rbdhyp2}) on the cut-off resolvent for $P=h^2\DX$ or $P=h^2\DX +V$ for certain potentials $V=V_0+h^2V_2$
holds.

\subsection{An example with a single connected end}\label{ss:singeex}
For $n\geq 2$ we can give $X=\R^n$ a warped product structure that 
makes it a manifold with an infinite cylindrical end.  Let $\rho$ be the 
radial coordinate on $\R^n$, and let $g_0= d\rho^2 +f(\rho)g_{\Sphere^{n-1}}$, where
$ g_{\Sphere^{n-1}}$ is the usual metric on the unit sphere $\Sphere^{n-1}$.  We 
assume $f\in C^\infty([0,\infty))$,  $f(\rho)=\rho^{2}$ in a 
neighborhood of $\rho=0$, the support of $f'$ is  $[0,\rho_0]$, with 
$f'(\rho)>0$ for $\rho\in (0,\rho_0)$.   The $d=2$ case
is illustrated by Figure \ref{f:cigar}.
  Then the only trapped 
geodesics are those which lie in a hypersurface $\{\rho=c\}$ for 
any $c\geq \rho_0$. 

\begin{figure}[h]
\includegraphics[width=0.6\textwidth]{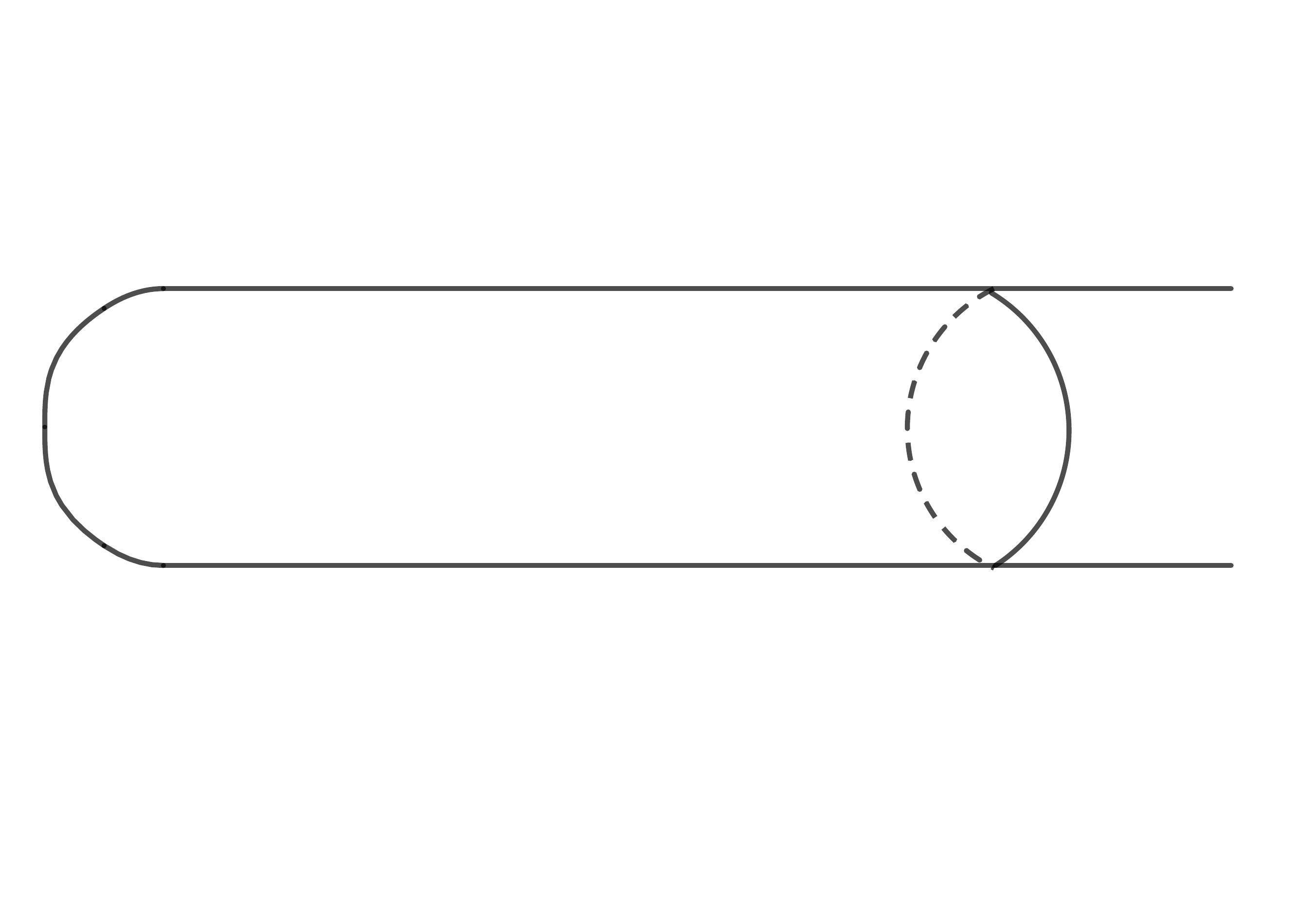} 
 \caption{A cigar-shaped, two-dimensional warped product.}\label{f:cigar}
\end{figure}

Let $g$ be any metric on $X$ so that  $g-g_0$ is supported in $\{(\rho,y) \mid  \rho<\rho_0\}$, and so that $g$ has the same trapped geodesics as $g_0$ does.
A discussion of constructing such metrics can be found in 
\cite[Example 1]{CDI}. We remark that there are  metrics satisfying these
conditions which are not 
rotationally symmetric.  

 For such 
manifolds $(X,g)$, $Y=\Sphere^{n-1} $ and all of $\calB$ 
is in the domain of the scattering map.

By \cite[Theorem 1.1]{CDI}, for any $\chi \in C_c^\infty(X)$,
$\|(\chi h^2\DX-1-i0)^{-1}\chi\|=O(h^{-2})$ when $h>0$ is sufficiently
small.  Hence  the estimate
 (\ref{eq:rbdhyp}) holds for $P=h^2\DX$ on
$(X,g)$, with $N_0=2 $.  
Moreover, by \cite[Theorem 3.1]{CDI}, (\ref{eq:rbdhyp}) holds
for $P=h^2\DX+V$ for a class of potentials $V\in C_c^\infty(X;\R)$.

We remark that in the case of a rotationally symmetric surface
these manifolds are very similar to, but not the 
same as, the surfaces considered 
in \cite{ZZ}.

\subsection{Examples modifying hyperbolic surfaces}\label{s:hs}
Starting with a convex cocompact
hyperbolic surface $(X,g_H)$, one can modify the metric
on the ends of the manifold $X$
 in such a way as to obtain a manifold with cylindrical ends so that
the cut-off resolvent is polynomially bounded.

\begin{figure}[h]
\hspace{2cm}
\labellist
\small
\pinlabel $r$ [l] at 1100 40
\pinlabel $\cosh^2\!r$ [l] at 880 440
\pinlabel $f(r)$ at 1150 300
\endlabellist

\includegraphics[width=5cm]
{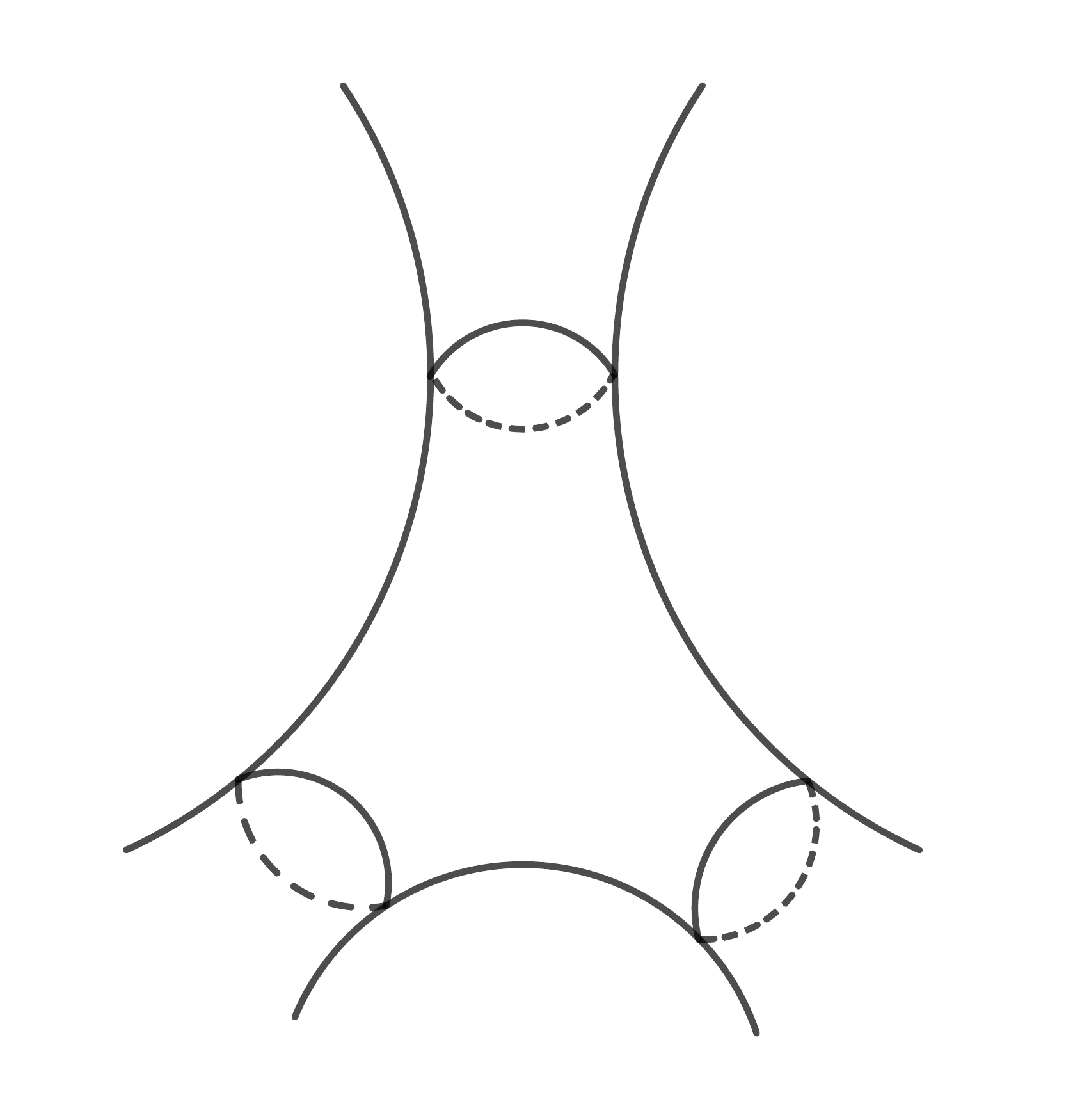}
\hspace{1cm}
\includegraphics[width=5.5cm]{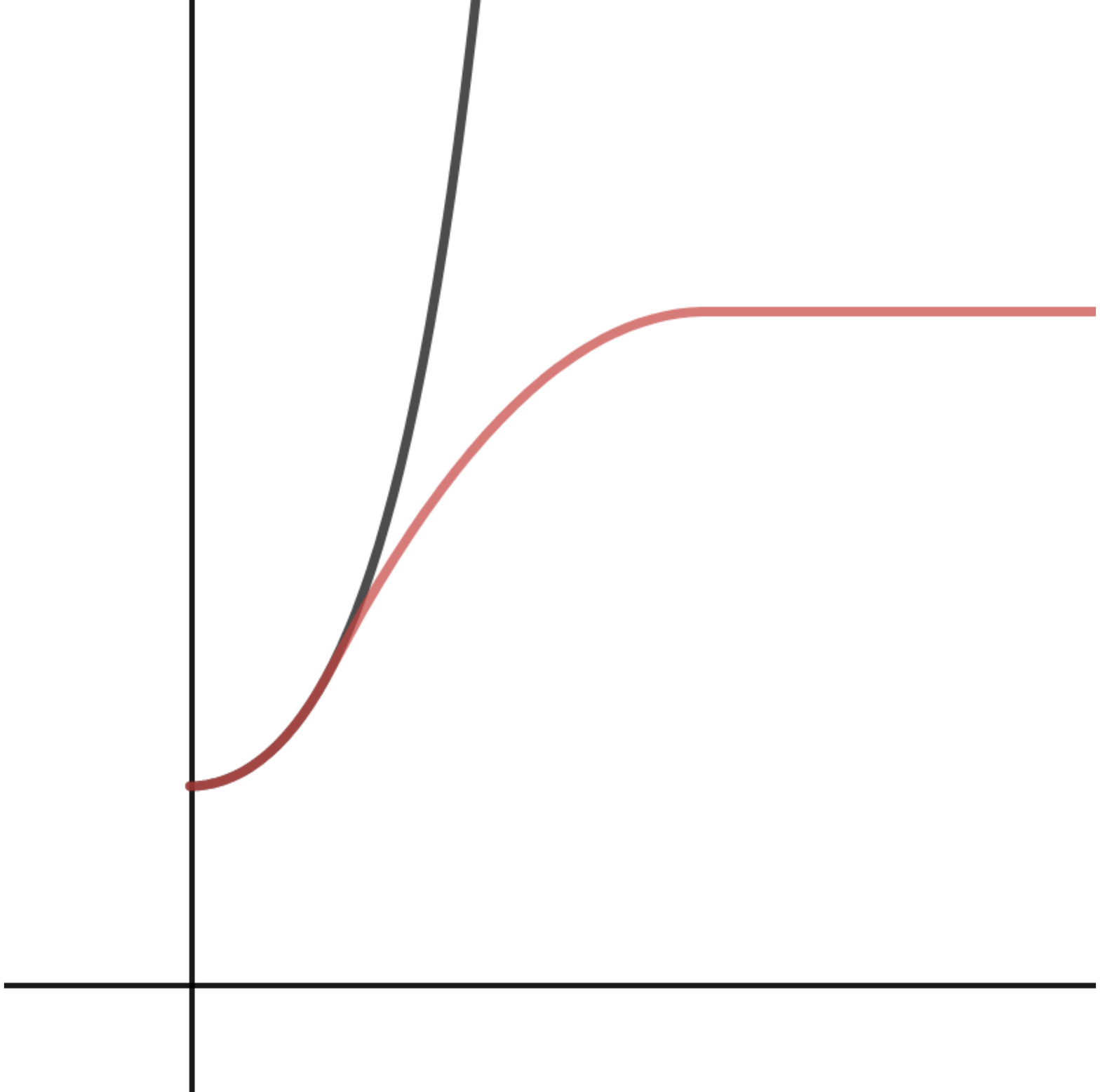}




 \caption{A hyperbolic surface $(X,g_H)$ with three funnels, and an example of a function $f$ satisfying the conditions on the warping function in Section \ref{s:hs}.}
\label{f:hyperbolic}
\end{figure}

There is a compact set $N\subset X$ so that 
$X\setminus N= (-4, \infty)_r\times Y_y$,
and $g_H \rest_{X\setminus N}=dr^2+\cosh^2(r+4)g_{Y}$.
Here we have modified slightly the usual convention to fit with our 
convention of using the coordinate $r\in (-4,\infty)$ on the ends of our manifold $X$.
The set $N$ is called the convex core of $X$, and the
 manifold $Y$ is the disjoint union of $k$ circles that might not 
have the same
length.    Let $f\in C^\infty(\R;(0,\infty))$ be equal to $\cosh^2r$
near $r=0$, with $f'$  compactly supported and $f'>0$ on the interior of the
convex hull of its support.
Let $g$ be the smooth metric on $X$ defined by
$g\rest_{N}=g_{H}\rest_N$, and $g\rest_{X\setminus N}=dr^2+f(r+4)g_Y$.  

Under these hypotheses, by \cite[Theorem 1.1]{CDI}  using results of \cite{bd,dyza}
the Laplacian $\DX$ on $(X,g)$ satisfies
$\|\chi (h^2\DX-1-i0)^{-1}\chi\|=O(h^{-2})$ for any $\chi \in C_c^\infty(X)$,
implying (\ref{eq:rbdhyp}) with $N_0=2$.  In fact the result is a bit stronger; see \cite[Section 3.3]{CDI} for further discussion
and references.




In higher dimensions
it is possible, but more complicated, to do a similar construction to
 modify the metrics on (certain) hyperbolic manifolds to give
manifolds with infinite cylindrical ends so that  
the resolvent of the semiclassical Laplacian satisfies (\ref{eq:rbdhyp});
 see \cite[Section 3.3]{CDI}.

\subsection{Right circular cylinder}\label{ss:rsc}

Set $X=\R_s\times \Sphere^{1}_y$, where $\Sphere^1$ is the unit circle,
and consider the product metric on $X$.
Let $W\in C_c^\infty(X;\R)$
 satisfy $W_0(s):=\int_0^{2\pi} W(s,y)dy \geq 0$, with 
$W_0\not \equiv 0$.  Then by \cite[Proposition 4.4 and Lemma 4.5]{Chr20} the operator $P=h^2\DX+h^2W$ 
satisfies (\ref{eq:rbdhyp}) with $N_0=2$.  Thus our results can 
be interpreted to give
results for the nonsemiclassical Schr\"odinger  operator at high energy.

In this case, the scattering map has as its domain all 
of $\calB$ and we can find the scattering map explicitly.  
Here $Y$ is the disjoint union of two circles, which
we write 
$Y=\Sphere_L^1\sqcup \Sphere_R^1$
for the cross sections of the connected ends of $X$ on
which
$s$ is bounded above  (for $\Sphere^1_L$; the ``left'' end) or
$s$ is bounded below (for $\Sphere^1_R$; the ``right'' end).  We use global coordinates $(s, y)\in \R \times[0,2\pi)$ on $\R \times \Sphere^1$, and use these same coordinates $y$ on $\Sphere^1_L$ and $\Sphere_R^1$.  Thus we can see in a particularly simple example how our choice
of function $r$ giving a coordinate on $\Xinf$ (or equivalently the decomposition $X=X_C\cup \Xinf$) affects
the scattering map.

Suppose $\supp (W)\subset [-a,a]\times\Sphere^1$, and set $X_C=[-a,a]\times 
\Sphere^1$.
Then the sets $\{\pm s=a+4\}$ correspond to the set $\{r=0\} \subset X$.
Recalling that $P=h^2\DX+h^2W$ here, a simple computation finds that
if $(y_-,\eta_-)\in T^*\Sphere^1_{R}\subset T^*Y$ with $|\eta_-|<1$
then $\kappa (y_-,\eta_-)=(y_+,\eta_-)$, where $y_+\in \Sphere^1_{L}$
and, modulo $2\pi$, $y_+=y_-+2(a+4)\eta_-/\sqrt{1-\eta_-^2}$.  A similar
computation works for points in $T^*\Sphere_L^1$.

\subsection{Warped products}\label{ss:wp}
Set $X=\R_s \times (Y_0)_y$ and $g= ds^2+(f(s))^{4/(n-1)}dg_{Y_0}$, where $(Y_0,g_{Y_0})$ 
is a smooth compact Riemannian manifold and $f\in C^\infty(\R; \R_+)$, with 
$f(s)=1$ if $|s|>a$.  We consider two special classes of functions $f$, which
give rise to 
manifolds with qualitatively different behavior both in terms of the trapped
geodesics and in terms of the number of embedded eigenvalues of $h^2\DX$.
For the first 
one (\ref{eq:rbdhyp}) (and hence also (\ref{eq:rbdhyp2})) holds for $P=h^2\DX$ (and $P=h^2\DX+V$ for some $V$), and 
for the second we show that (\ref{eq:rbdhyp2}) holds for $P=h^2\DX$. 

Here $Y$ is the disjoint union of two copies of
$Y_0$.  We write $Y=Y_{0L}\sqcup Y_{0R}$, where 
$Y_{0L}$ and $Y_{0R}$ are copies of $Y_0$ identified with the 
cross section of the ``left'' and ``right'' ends
of $X$, respectively. 

\subsubsection{Hourglass-type warped products}\label{ss:hourglass}
In addition to the assumptions made on $f$ above, assume that $f$ has a 
single critical point in $(-a,a)$, and it is a nondegenerate minimum.  The surface on the left in Figure \ref{f:bulge} provides an example. Then
by \cite[Theorem 3.1]{CDI}, see \cite[Section 3.4]{CDI}, for 
any $\chi \in C_c^\infty(X)$,
\begin{equation}\label{eq:rbdchda}
\|\chi (h^2\DX-1-i0)^{-1}\chi\|=O(h^{-2})\; \text{ for $h>0$ sufficiently
small.}
\end{equation}  Thus  the 
estimate (\ref{eq:rbdhyp}) holds with $N_0=2$
for $P=h^2\DX$.  
We note that 
the estimate (\ref{eq:rbdchda}) implies that $\DX$ has only finitely many eigenvalues.
Each trapped geodesic on this manifold lies in a set $\{s=c\}$, for some $c\in \bbR$ with 
 $f'(c)=0$. 
 
For  Schr\"odinger operators  $P=h^2\DX+V$, where $V\in C_c^\infty(X;\R)$
satisfies certain conditions the estimate 
(\ref{eq:rbdchda}) holds,  see \cite[Theorem 3.1]{CDI}.  For example, if $V_2\in C_c^\infty(X;\R)$, and $V(x)=V(x,h)=h^2V_2(X)$, then (\ref{eq:rbdchda}) holds.
For this example, because we use
the results of \cite{CDI} to prove the estimate (\ref{eq:rbdhyp}),
the potentials $V$ need not be functions of $s$ alone.

We now return to the case $P=h^2\DX$.  Let $f_m$ be the minimum value of $f$, and 
set $|\eta|_c=f_m^{4/(n	-1)}$.  Then
using properties of geodesics on warped products,
 the domain of the scattering map is $\{(y,\eta)\in T^*Y  \mid  |\eta|<1\; \text{and}\; |\eta|\not = |\eta|_c\}$.  Suppose $(y_-,\eta_-)\in T^*Y_{0L}$.
If $|\eta_-|<|\eta|_c$, then $\kappa(y_-,\eta_-)\in T^*Y_{0R}$, while
if $|\eta|_c<|\eta_-|<1$, $\kappa(y_-,\eta_-)\in T^*Y_{0L}$.

We introduce some notation to describe one consequence of this for the scattering
matrix.  Let $\prj_L:L^2(Y_{0L}\sqcup Y_{0R})\rightarrow L^2(Y_{0L})$
and $\prj_R:L^2(Y_{0L}\sqcup Y_{0R})\rightarrow L^2(Y_{0R})$ be the natural
orthogonal projections.  Then if $\psi_s\in C_c^\infty(\R)$ is 
supported in $(-\infty, |\eta|^2_c)$, it follows from the 
mapping properties of $\kappa$ and Theorem 
\ref{thm:v1}
that $\| \prj_L \tS \psi_s(h^2\DY)\prj_L \|=O(h^\infty).$
Likewise, if $\psi_l\in C_c^\infty(\R)$ is supported in 
$(|\eta|^2_c,1)$, then $\| \prj_R \tS \psi_l(h^2\DY)\prj_L \|=O(h^\infty)$.

Of course, there are similar results focusing on right multiplication
by  $\prj_R$ rather
than $\prj_L$.

\subsubsection{Warped products with bulges} \label{ss:bulge} Now consider what is in some 
sense the opposite situation to that of Section \ref{ss:hourglass}: in addition to the general assumptions on $f$ in
Section \ref{ss:wp}, assume that $f$ has a single critical point in $(-a,a)$, and 
it is a maximum.  In Figure \ref{f:bulge}, the figure on the left
illustrates the hourglass-type warped
products of Section \ref{ss:hourglass}, while that on the right illustrates
the warped products with bulges discussed in  this section.

\begin{figure}[h]
\hspace{3cm}

\includegraphics[width=5.5cm]{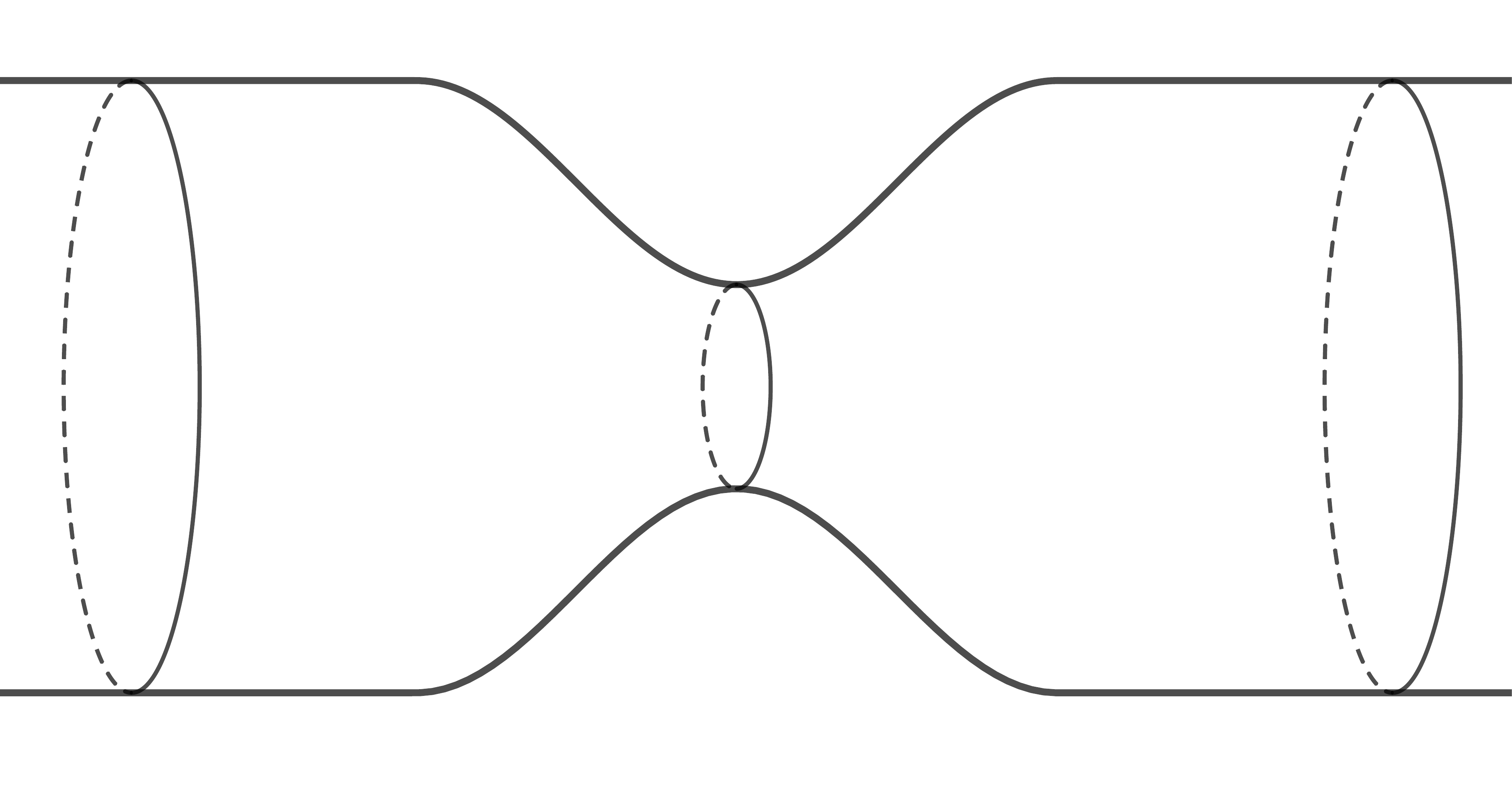} 
\hspace{1cm}
\includegraphics[width=5.5cm]{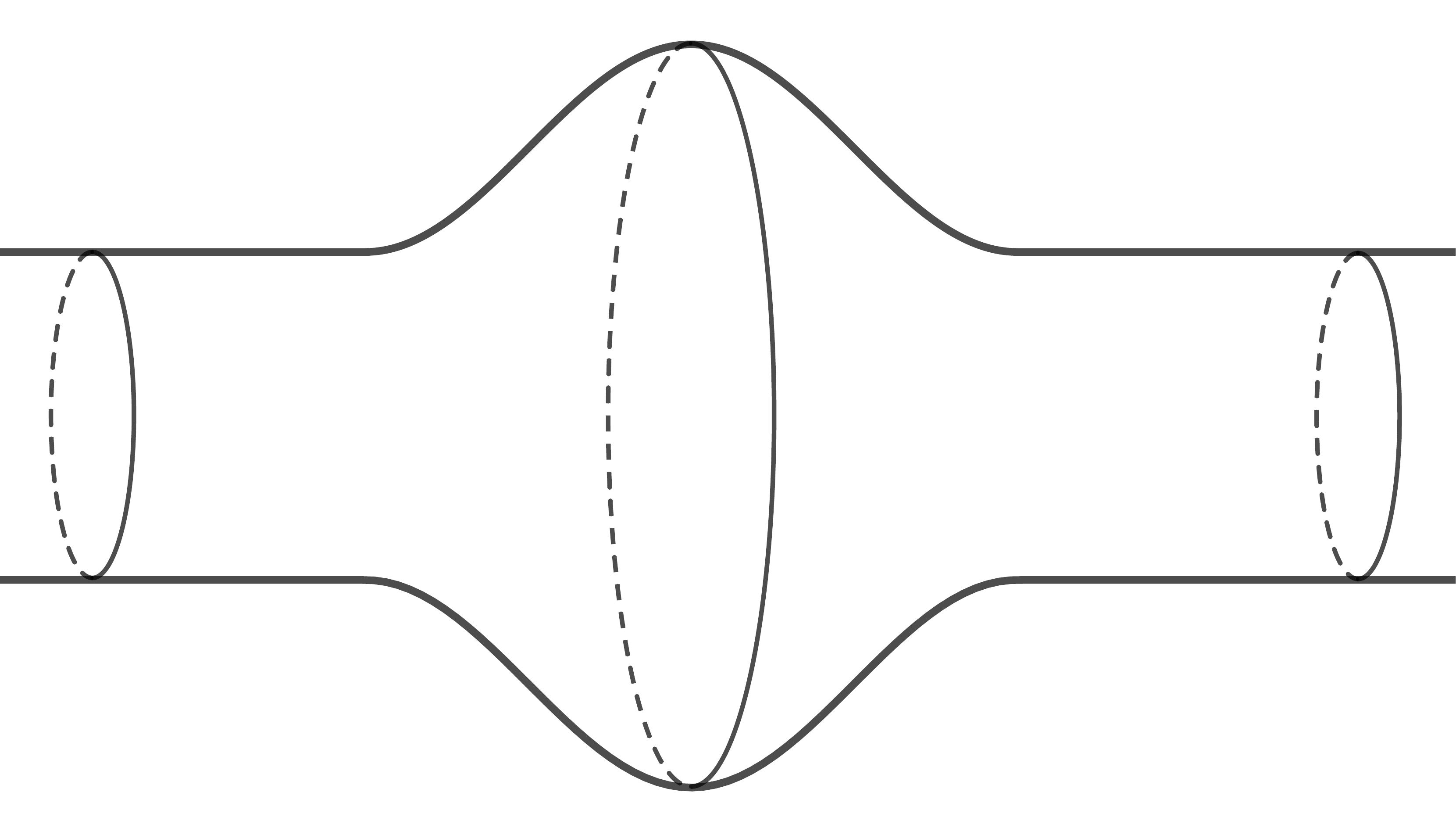}

 \caption{An hourglass-shaped  warped product (left), and 
a warped product with a bulge (right).}\label{f:bulge}
\end{figure}


With these assumptions on $f$, $X$ has infinitely many trapped geodesics that lie entirely in 
the region with $s\in (-a,a)$, and
 it is straightforward to show via a separation
of variables and results from semiclassical analysis that $\DX$ has infinitely many eigenvalues accumulating at infinity,
\cite{CZ, par}.  Hence if $\chi \in C_c^\infty(X)$
is nontrivial, then there is a sequence $\{h_j\}$ tending to $0$ so
that 
$ \lim_{\epsilon \downarrow 0}\|\chi (h_j^2\DX-1-i\epsilon)^{-1}\chi\|=\infty$.
Nonetheless, we show in  Lemma \ref{l:bulgeestimate} that
 (\ref{eq:rbdhyp2}) holds with $N_0=1$.  In comparison 
with the example of Section \ref{ss:hourglass}, the estimate is improved:
$N_0=1$ here, compared to $N_0=2$ in the hourglass-type example.  This 
may be surprising, since the trapping in the warped products with bulges
is stronger than that in the hourglass-type warped products.  This 
might be attributed to the fact that our microlocal cutoff in the cross-section,
$\mathbbold{1}_{[0,1-\epsilon]}(h^2\DY)$ has the effect of cutting off away from trapped
bicharacteristics in $\{p=1\}$ in the examples of this section, but
not the examples of Section \ref{ss:hourglass}.  Alternatively, the difference may be 
an artifact of the proof.

In this setting, the domain of the scattering map $\kappa$ is all of $\calB$.
Using the notation of Section \ref{ss:hourglass}, if $(y_-,\eta_-)\in \calB\cap T^*Y_{0L}$, then $\kappa(y_-,\eta_-)\in T^*Y_{0R}$.  Then 
by Theorem \ref{thm:v2}  for any
$\psi \in C_c^\infty([-1,1)),$ $\|\prj_{0L}\tS \psi(h^2\DY)\prj_{0L}\|=O(h^\infty)$.

For the special case of a surface of revolution with a bulge we compute the scattering map in Section \ref{ss:smwp}.

\section{Existence of the Poisson operator and the scattering matrix}\label{s:Posm}

In this section we discuss the Poisson operator, introduced in Section \ref{s:idea}, and prove some consequences for the scattering matrix, Lemmas \ref{l:tSwelldefined} and \ref{l:Sest}.  
The construction we give of the Poisson operator in this section is different from the more microlocal construction that we will give in Section \ref{s:gef}.
Much of the content of this section is known, see e.g. \cite{tapsit, chr95, par, lrb, CD2}, but we include it for the reader's convenience.

We begin by checking that there is an operator satisfying the conditions given in Definition \ref{d:pop} to define the Poisson operator, and that this uniquely 
determines the operator.  Recall that we assume $1\not \in \spec(h^2\DY)$.
Let $\mathcal{PR}_1$ denote orthogonal projection onto the eigenfunctions of $P$ with eigenvalue $1$, with $\mathcal{PR}_1=0$
if $1$ is not an eigenvalue of $P$.  Then it follows from \cite[Section 6.8]{tapsit} or \cite[Lemmas 2.2 and 2.3]{CD2}
that $\lim_{\epsilon \downarrow 0} \langle{r}\rangle ^{-(1/2+\delta)} (P-1\pm i \epsilon)^{-1}\left(I- \mathcal{PR}_1\right)\tilde{\chi}$
is a bounded operator on $L^2(X)$ for any $\tilde{\chi}\in L^\infty_c(X)$ and $\delta>0$.

Let $\varphi\in C^\infty(\bbR;[0,1])$ satisfy 
 $\varphi(r)=1$ for $r\geq 0$, and $\varphi(r)=0$ for $r\leq -1/2$.
Given $f\in L^2(Y)$,
 set 
\begin{equation}\label{eq:FXtilde}
\gef_{\Xtil}(r,y)=e^{-ir(I-h^2\DY)_+^{1/2}/h}\mathbbold{1}_{[0,1]}(h^2\DY)f \in \langle r \rangle ^{1/2+\delta } H^2(\R \times Y)
\end{equation}
and
\begin{equation}\label{eq:Finf}
\gef_{\Xinf}(r,y)=
\varphi (r) \gef_{\Xtil}\in \langle r \rangle ^{1/2+\delta}H^2(X).
\end{equation}
Then 
\[
(\Ph-1) \gef_{\Xinf}=(h^2\DX-1) \gef_{\Xinf} = \left(-h^2\varphi''(r) - 2h^2\varphi'(r)\partial_r 
\right)\gef_{\Xtil}
\]
has compact support on $\Xinf \subset X$.  Moreover, this function is orthogonal to any eigenfunction of $P$ with eigenvalue $1$.  This is because a separation of variables argument shows that if $g\in L^2(X)$ satisfies $(P-1)g=0$, 
  then $\mathbbold{1}_{[0,1]}(h^2\DY)(g\rest_{\Xinf})=0$.

 We now set
 \begin{equation}
 \label{eq:popconstruct}\pop f = \gef_{\Xinf} - (\Ph-1-i0)^{-1}(h^2\DX-1)\gef_{\Xinf}= \gef_{\Xinf} -\lim_{\epsilon\downarrow 0} (\Ph-1-i\epsilon)^{-1}(h^2\DX-1)\gef_{\Xinf}
 \end{equation}
 and check that it satisfies the requirements on $\pop f$ made in the definition of the Poisson operator.  Note that by construction, $(\Ph-1)\pop f=0$ and
 if $g\in L^2(X)$  satisfies $(P-1)g=0$, then $\langle \pop f, g\rangle =0$.  Since 
 \begin{equation*}\left( (\Ph-1-i0)^{-1}(h^2\DX-1)\gef_{\Xinf}\right)\rest_{\Xinf}\\
 =e^{ir \sro /h}f_+
 \end{equation*}
 for some function $f_+\in L^2(X)$, we have shown that $\pop f\in \langle r \rangle ^{1+\delta}H^2(X)$, and  thus have
 shown that there is an operator satisfying the conditions of Definition \ref{d:pop}.

 Next we consider uniqueness.   Suppose there are two such Poisson operators, $\pop$ and $\tilde{\pop}$.  For incoming data $f\in L^2(Y)$, denote the corresponding outgoing 
 data by $f_+$ and $\tilde{f}_+$, respectively.   
 Let $\{\phi_j\}$ be a complete
set of orthonormal eigenfunctions of $\DY$, with 
\begin{equation}\label{eq:phis} \DY\phi_j=\sigma_j^2\phi_j, \; \text{and} \; 0=\sigma_1^2\leq \sigma_2^2\leq \cdots.
\end{equation}
We have $(P-1)(\pop-\tilde{\pop})f=0$ and 
 \begin{equation}
 \left( (\pop-\tilde{\pop})f\right)\rest_{\Xinf}=e^{ir \sro /h}(f_+-\tilde{f}_+)= \sum_{j}c_j e^{ir(1-h^2\sigma_j^2)^{1/2}/h}\phi_j
 \end{equation}
 for some  $\{c_j\}$, $c_j \in \C$.  Here we use the convention that $ (1-h^2\sigma_j^2)^{1/2}$ has nonnegative real and imaginary parts, as in the definition of 
 $(I-h^2\DY)^{1/2}$.  Applying a Stokes' identity  gives
 $$0= \lim_{R\rightarrow \infty} \int_{X:r<R} \Big( (P-1)(\pop-\tilde{\pop})f \Big) \overline{(\pop-\tilde{\pop})f}\;dx= -\frac{2i}{h}\sum_{h^2\sigma_j^2<1} \sqrt{1-h^2\sigma_j^2}|c_j|^2$$
 implying that $c_j=0$ if $h^2\sigma_j^2<1$, so  that $(\pop-\tilde{\pop})f\in L^2(X)$.  Thus $(\pop-\tilde{\pop})f$ is an $L^2$ element of  the null space of $P-1$, and hence is either $0$ or an eigenfunction.
 But $\langle  (\pop-\tilde{\pop})f, g \rangle =0$ for any eigenfunction $g$ of $P$ with eigenvalue $1$, so $(\pop-\tilde{\pop})f\equiv 0$ and the Poisson operator is uniquely defined.
 
 We remark that our argument above shows that if we omit the requirement $\langle  \pop f,g\rangle =0$ for every $g\in L^2(X)$ in the null space of $P-1$, then $\pop f$ is determined up to addition of an eigenfunction of $P$ with eigenvalue $1$.
 We shall use this below.

\vspace{2mm}
\noindent
{\em Proof of Lemma \ref{l:tSwelldefined}.}  
For $f\in \mch_Y$, the function $F=\pop f$ satisfies the conditions of the first part of the Lemma.  Our argument above shows that if 
$F,\; \tilde{F}\in \langle r \rangle ^{1/2+\delta}H^2(X)$ with $(P-1)F=0=(P-1)\tilde{F}$, and both $F$ and $\tilde{F}$ have an expansion as 
in (\ref{eq:uexp})  with the same incoming data $f=f_-=\tilde{f}_-$, then $F-\tilde{F}$ is in the $L^2$ null space of $P-1$.   Hence $F-\tilde{F}$ is either $0$ or an 
$L^2$ eigenfunction.  Since for any $L^2$ eigenfunction $g$ with eigenvalue $1$, $\bbI(h^2\DY)(g\rest_{\Xinf})=0$, this ensures $\bbI(h^2\DY)(f_+-\tilde{f}_+)=0$, where
$f_+$, $\tilde{f}_+$ are the outgoing data as in (\ref{eq:uexp}) for $F$ and $\tilde{F}$, respectively.   Thus the scattering matrix is well-defined when $1\not \in \spec(h^2\DY)$.


We next relate the scattering matrix $\tS$ defined above to those of \cite[Section 1.3]{chr95} and \cite{par}; see also \cite[Section 6.10]{tapsit},
again under the assumption that $1\not \in \spec(h^2\DY)$.
 Suppose $h^2 \sigma_k^2<1$.  Then we can write the expansion as in  (\ref{eq:poponend}) for $\pop \phi_k$ as 
$$(\pop \phi_k)\rest_{\Xinf}(r,y)=e^{-ir(1-h^2\sigma_k^2)^{1/2} /h}\phi_k+ \sum_j e^{ir(1-h^2\sigma_j^2)^{1/2}/h}\Ss_{jk}\phi_j$$
 where $\Ss_{jk}= \Ss_{jk}(h)$ are some scalars.
This uniquely determines the $\Ss_{jk}$ if $h^2\sigma_j^2<1$.   Comparing our definition of $\tS$, we find
$\tS \phi_k=\sum_{h^2\sigma_j^2\leq 1} \Ss_{jk}\phi_j$.   Moreover, this shows that the scattering matrix of  \cite[Definition 1.3]{chr95} is 
 $S_U=(I-h^2\DY)^{1/4}_+ \tS (I-h^2\DY)^{-1/4}_+$; this operator is unitary on $\mch_Y$.

   Combining this with results of \cite[Section 1.3]{chr95} shows that if $h^2\sigma_j^2=1$ for some $j$, then
$\lim_{h'\uparrow h} \tS(h')$ exists as a bounded operator.
\qed

We shall later need a bound on $\|S\|$.  
We recall that the operator $(1-h^2\DY)^{1/4}_+ \tS (1-h^2\DY)^{-1/4}_+$ is unitary on $\mch_Y=\mathbbold{1}_{[0,1]}(h^2\DY)L^2(Y)$.  This does not immediately give a good bound on 
$\| \tS \|$ itself, since $\|(I-h^2\DY)_+^{-1/2}\|$ is large when $1/h^2$ is near an eigenvalue of $\DY$. 
 However, under the assumptions of Theorem \ref{thm:v2}, $h^2\DY$ commutes with $\tS$, and in this setting $\|\tS \|=1$.    In general we have the following lemma.  Recall
 $\mathcal{PR}_1$ is  orthogonal projection onto the eigenfunctions of $P$ with eigenvalue $1$.
\begin{lemma}\label{l:Sest}
There is a $C>0$ independent of $h$ so that 
$$\| \tS\| \leq C h\|\mathbbold{1}_{[0,1]}(h^2\DY)\mathbbold{1}_{[0,1]}(r)(\Ph-1-i0)^{-1}(I-\mathcal{PR}_1) \mathbbold{1}_{[-1,0]}(r)\|.$$
\end{lemma}
\begin{proof} First suppose $1\not \in \spec(h^2\DY)$.
We use the Poisson operator $\pop$ as constructed in (\ref{eq:popconstruct}).  Let $f\in \mch_Y$, and denote the outgoing data in $\pop f$ by 
 $f_+$.   We note from our construction of $\pop$ that 
 \begin{multline}
  \left( e^{ir\sro/h} \mathbbold{1}_{[0,1]}(h^2\DY) f_+ \right)\rest_{ \Xinf:r>1} =  \mathbbold{1}_{[0,1]}(h^2\DY)\left( (\Ph-1-i0)^{-1} [h^2\partial^2_r,\varphi]F_{\tX}\right)\rest_{\Xinf:r>1}
  \\= \mathbbold{1}_{[0,1]}(h^2\DY)\left( (\Ph-1-i0)^{-1} (I-\mathcal{PR}_1)[h^2\partial^2_r,\varphi]F_{\tX}\right)\rest_{\Xinf:r>1}.
  \end{multline}
 where $F_{\tX}$ is defined in (\ref{eq:FXtilde}) and $\varphi$ is as in (\ref{eq:Finf}).  Thus
\begin{align*}& 
\|\mathbbold{1}_{[0,1]}(h^2\DY)f_+\|_{L^2(Y)}\\ & \leq  \|\mathbbold{1}_{[0,1]}(h^2\DY)\mathbbold{1}_{[0,1]}(r)(\Ph-1-i0)^{-1}(I-\mathcal{PR}_1) \mathbbold{1}_{[-1,0]}(r)\|_{L^2(X)\rightarrow L^2(X)}
\| [h^2\partial_r^2,\varphi]F_{\tX}\|_{L^2(X)}  \\ & 
\leq C h\|\mathbbold{1}_{[0,1]}(h^2\DY)\mathbbold{1}_{[0,1]}(r)(\Ph-1-i0)^{-1}(I-\mathcal{PR}_1) \mathbbold{1}_{[-1,0]}(r)\|_{L^2(X)\rightarrow L^2(X)}\| f\|_{L^2(Y)}
\end{align*}
for some constant $C$.

To handle the case of $1\in \spec(h^2\DY)$, take the limit as $h'\uparrow h$ to obtain the desired bound.
\end{proof}

\section{The resolvent on $\Xtil=\R\times Y$}
\label{s:Xtilops}

We shall need some facts about the behavior of the resolvent of the Laplacian
on the  manifold $\Xtil=\bbR \times Y$ with the product
metric, and the operators that 
arise when studying it.    We begin with a simple lemma about the resolvent
of $-h^2\partial_r^2$ on $\bbR$.
\begin{lemma}\label{l:1d}
Let $\tau>0$ and $f\in L^2_c(\bbR)$ be supported in the interval $[a,b]$.
Then, if $r>b$
$$\left((-h^2\partial_r^2-\tau^2 \mp i0)^{-1}f\right)(r)
= \pm \frac{i}{2\tau h}e^{\pm i \tau r/h}\int_{a}^b e^{\mp i \tau r'/h}f(r')dr'$$
and 
 $$\left((-h^2\partial_r^2+\tau^2)^{-1}f\right)(r)
= \frac{1}{2\tau h}e^{- \tau r/h}\int_{a}^b e^{ \tau r'/h}f(r')dr'.$$
If $r<a$, then
$$\left((-h^2\partial_r^2-\tau^2 \mp i0)^{-1}f\right)(r)
= \pm \frac{i}{2\tau h}e^{\mp i \tau r/h}\int_{a}^b e^{\pm i \tau r'/h}f(r')dr'$$
and 
$$\left((-h^2\partial_r^2+\tau^2)^{-1}f\right)(r)
= \frac{1}{2\tau h}e^{\tau r/h}\int_{a}^b e^{ -\tau r'/h}f(r')dr'.$$
\end{lemma}
\begin{proof}
We use that, for $\lambda \in \bbC$ with $\Im \lambda >0$,
$$\left((-h^2\partial_r^2-\lambda^2 )^{-1}f\right)(r)
= \frac{i}{2\lambda h}\int_{-\infty}^\infty e^{i\lambda |r-r'|/h}f(r')dr'.$$
Then the lemma follows directly using the support properties of $f$.
\end{proof}




Define  operators
$T_{\pm}:L_c^2(\Xtil)\rightarrow L^2(Y)$ by
$$ T _{\pm} f = \int_{\R}
e^{\mp ir' (\sro)/h}f(r',\bullet)dr'.$$

\begin{lemma}\label{l:prodresolve}
Let $f\in L^2_c(\Xtil)$ be supported in $[a,b]\times Y$, and let
$\psp\in C_c^\infty([0,1))$.  
Then
\begin{align}\label{eq:prodresolve2}
\left(\psp(h^2\DY)\sro(h^2\DXt -(1\pm i0))^{-1}f\right)\rest_{ r>b}
& = \frac{\pm i}{2h}\psp(h^2\DY)e^{\pm i r \sro/h}T_\pm f
\end{align}
Moreover,
\begin{equation}\label{eq:prodresolve3}
\left(\psp(h^2\DY)\sro (h^2\DXt -1-i0)^{-1}f\right)\rest_{ r<a}
= \frac{i}{2h}\psp(h^2\DY)e^{-i r \sro /h}T_- f.
\end{equation}
\end{lemma}
\begin{proof}
The proof uses separation of variables and the spectral theorem.

Let $\{\sigma_j^2\}$, $\{\phi_j\}$ be the eigenvalues and eigenfunctions
of $\DY$ as in (\ref{eq:phis}).  Define
$\tau_j=\tau_j(h)=(1-h^2\sigma_j^2)^{1/2}$, where our convention is
 the square root has nonnegative
real and imaginary parts.  Then writing 
$f(r,y)=\sum_{j=1}^\infty f_j(r)\phi_j(y)$, we have
$$\left((I-h^2\DY)^{1/2}(h^2\DXt -1-i0)^{-1})f\right)(r,y)=
\sum_{j=1}^\infty \tau_j\left((-h^2\partial_r^2-\tau_j^2 -i0)^{-1}f_j\right)(r)\phi_j(y).$$
Now we assume that $r>b$. Then from Lemma \ref{l:1d},
$$\left((I-h^2\DY)^{1/2}(h^2\DXt -1-i0)^{-1}f\right)(r,y)=
\sum_{j=1}^\infty \frac{i}{2 h}e^{ i \tau_j r/h} 
\left( \int_{a}^b e^{- i \tau_j r'/h}f_j(r')dr'\right) \phi_j(y)$$
so that
\begin{multline}
\left(\psp(h^2\DY)(I-h^2\DY)^{1/2}(h^2\DXt -1-i0)^{-1}f\right)(r,y)\\=
\sum_{j=1}^\infty \frac{i}{2 h}\psp(h^2 \sigma_j^2)e^{ i \tau_j r/h} 
\left( \int_{a}^b e^{- i \tau_j r'/h}f_j(r')dr'\right) \phi_j(y).
\end{multline}
But then, for the top choice of sign,
 this is the representation of the operator on the right hand
of  (\ref{eq:prodresolve2}) given by the spectral theorem.

The proofs of the remaining equalities are similar.
\end{proof}


\section{Resolvent estimates on $X$} \label{s:reX}
This section contains two lemmas that we use later to allow some flexibility in exactly how we cut-off the resolvent on the end $\Xinf$.  We recall that $r>-4$ only on the end 
$\Xinf$.  These estimates do not require the bounds (\ref{eq:rbdhyp}) or  (\ref{eq:rbdhyp2}).

\begin{lemma}\label{l:reX1}
Let $c>-4$, $M>c$ and $z\in \C\setminus \operatorname{spec}(P)$.  Then 
$$\| \mathbbold{1}_{[M, M+1]}(r)(\Ph-z)^{-1}\mathbbold{1}_{(-\infty,c]}(r)\|\leq \| \mathbbold{1}_{[c,c+1]}(r)(\Ph-z)^{-1}\mathbbold{1}_{(-\infty,c]}(r)\|$$
and for any $\epsilon>0$
\begin{multline*}
\|\mathbbold{1}_{[0,1-\epsilon]}(h^2\DY) \mathbbold{1}_{[M, M+1]}(r)(\Ph-z)^{-1}\mathbbold{1}_{(-\infty,c]}(r)\|
\\ \leq \| \mathbbold{1}_{[0,1-\epsilon]}(h^2\DY) \mathbbold{1}_{[c,c+1]}(r)(\Ph-z)^{-1}\mathbbold{1}_{(-\infty,c]}(r)\|.
\end{multline*}
\end{lemma}
\begin{proof}
Recall that on the end $\Xinf$, $P=-h^2\partial_r^2+h^2\DY$.  We use here the notation $\tau_j(z)=\tau_j(z,h)=(z-h^2\sigma_j^2)^{1/2}$, where the square root has positive imaginary part, 
which is possible since $z\not \in [0,\infty)$.  Then for any $f\in L^2(X)$ there are $c_j\in \C$ so that 
$$(P-z)^{-1} ( \mathbbold{1}_{(-\infty,c]}(r) f)\upharpoonleft_{r>c} (r,y)=\sum _{j=1}^\infty c_j e^{i \tau_j(z)r/h} \phi_j(y).$$
Since for each $j$, $|e^{i \tau_j(z)r/h}|$ is exponentially decreasing as $r\rightarrow \infty$, 
$$\| \mathbbold{1}_{[M, M+1]}(r) \sum _{j=1}^\infty c_j e^{i \tau_j(z)r/h} \phi_j(y) \|^2 \leq \| \mathbbold{1}_{[c, c+1]}(r) \sum _{j=1}^\infty c_j e^{i \tau_j(z)r/h} \phi_j(y) \|^2$$
proving the first statement of the lemma.

The proof of the second statement is very similar.
\end{proof}

\begin{lemma} \label{l:reX2}
For $M>0$ there is a $C=C(M)>0$ so that for all $z\in \C\setminus \operatorname{spec}(P)$, $h\in (0,1]$
$$\|\mathbbold{1}_{[0, 1]}(r)(\Ph-z)^{-1}\mathbbold{1}_{(-\infty,M]}(r)\| \leq C\left(h^{-2} + \|\mathbbold{1}_{[0, 1]}(r)(\Ph-z)^{-1}\mathbbold{1}_{(-\infty,0]}(r)\|\right).$$
Moreover, for every $M,\; \epsilon>0$ there is a $C=C(M,\epsilon)$ so that 
\begin{multline*}
\|\mathbbold{1}_{[0,1-\epsilon]}(h^2\DY) \mathbbold{1}_{[0, 1]}(r)(\Ph-z)^{-1}\mathbbold{1}_{(-\infty,M]}(r)\| 
\\ \leq C\left(h^{-2} + \|\mathbbold{1}_{[0,1-\epsilon]}(h^2\DY)\mathbbold{1}_{[0, 1]}(r)(\Ph-z)^{-1}\mathbbold{1}_{(-\infty,0]}(r)\|\right).
\end{multline*}
\end{lemma}
\begin{proof}
We can write 
\begin{equation}\label{eq:easysplit}
\mathbbold{1}_{[0, 1]}(r)(\Ph-z)^{-1}\mathbbold{1}_{(-\infty,M]}(r)= \mathbbold{1}_{[0, 1]}(r)(\Ph-z)^{-1}\mathbbold{1}_{(-\infty,0]}(r) + \mathbbold{1}_{[0, 1]}(r)(\Ph-z)^{-1}\mathbbold{1}_{[0,M]}(r).
\end{equation}
It is immediate that the first term on the right is bounded as desired, so we need only bound $\| \mathbbold{1}_{[0, 1]}(r)(\Ph-z)^{-1}\mathbbold{1}_{[0,M]}(r)\|$.

Let $P_{0D}$ denote the operator $-h^2\partial_r^2+h^2\DY$ on the product manifold $((-4,\infty)\times Y, (dr)^2+g_Y)$ with Dirichlet boundary conditions at $\{r=-4\}$.  Using $\tau_j(z)$ as in the previous lemma,
if $f\in L^2((-4,\infty)\times Y)$ and we write $f(r,y)=\sum_{j=1}^\infty f_j(r)\phi_j(y)$,  then for $z\in \C \setminus [0,\infty) $
$$((P_{0D}-z)^{-1} f)(r,y)=\sum_{j=1}^\infty \frac{i}{2\tau_j(z)h}\left( \int_{-4}^\infty (e^{i |r+4-r'|\tau_j(z) /h}- e^{i (r+r'+4)\tau_j(z) /h}) f_j(r')dr'\right)\phi_j(y).$$
Note that the choice of Dirichlet boundary condition ensures that for any $M'\in \R$,
\begin{equation}\label{eq:P0Dest}
\| \mathbbold{1}_{[-4, M']}(r)(P_{0D}-z)^{-1}\mathbbold{1}_{[-4,M']}(r)\| + \| \mathbbold{1}_{[-4, M']}(r)h\partial_r (P_{0D}-z)^{-1}\mathbbold{1}_{[-4,M']}(r)\| \leq \frac{C'}{h^2}
\end{equation}
for some $C'=C'(M')$, independent of $z\in \C\setminus[0,\infty)$ and $h\in(0,1]$.  

Now choose $\chi \in C_c^\infty(\R)$ so that $\chi(r)=1$ for $r\leq -1$ and $\chi(r)=0$ for $r\geq 0$.  Then, with $\chi=\chi(r)$,
$$(P-z)(1-\chi) (P_{0D}-z)^{-1}\mathbbold{1}_{[0,M]}(r) = \mathbbold{1}_{[0,M]}(r) -[P,\chi] (P_{0D}-z)^{-1}\mathbbold{1}_{[0,M]}(r)$$
so that
\begin{equation}\label{eq:mixedresolvents}
(P-z)^{-1} \mathbbold{1}_{[0,M]}(r)= (1-\chi) (P_{0D}-z)^{-1}\mathbbold{1}_{[0,M]}(r)+(P-z)^{-1}(h^2\chi''+2h^2 \chi'\partial_r)(P_{0D}-z)^{-1}\mathbbold{1}_{[0,M]}(r).
\end{equation}
Since $\chi'$ is supported in $[-1,0]$, using  (\ref{eq:P0Dest}) and (\ref{eq:mixedresolvents}) proves $\|\mathbbold{1}_{[0, 1]}(r)(\Ph-z)^{-1}\mathbbold{1}_{[0,M]}(r)\|$ is bounded as desired, completing 
the proof of the first part of the lemma.

To prove the second statement, we left multiply both sides of (\ref{eq:easysplit}) and (\ref{eq:mixedresolvents}) by $\mathbbold{1}_{[0,1-\epsilon]}(h^2\DY)$, and proceed as before.
\end{proof}
We remark that although Lemma \ref{l:reX1} and \ref{l:reX2} are stated for  the resolvent 
$(P-z)^{-1}$ for $z\in \C \setminus \operatorname{spec}(P)$, by a limiting argument they also hold for 
$(P-z\pm i0)^{-1}$, when   $z\in [0,\infty)$.  Of course,
if $z\in[0,\infty)$ the estimates are only 
meaningful if the right hand side is finite.

\section{Microlocal properties of components of the scattering matrix}\label{s:mp}


	
	


\mbox{}

In this section we analyze  operators that go into the approximation to the scattering matrix,
proving that they are Fourier integral operators (FIOs).
In dealing with canonical relations between cotangent bundles, it will be convenient to use the 
following notational principles:

\begin{itemize}
	\item We will identify the cotangent bundle of a Cartesian product with the product of
	the cotangent bundles, and separate points in cotangent bundle factors by a semi-colon.
	For example, $(r,\rho\,;\,y,\eta)\in T^*(\bbR\times Y)$ denotes the generic point in
	$T^*(\bbR\times Y)$ with $(r,\rho)\in T^*\bbR$ and $(y,\eta)\in T^*Y$.   This differs from some notation in the introduction.
	
	\item On occasion we will use the notation $\x = (x,\xi)\in T^*X$, $\y= (y,\eta),\,
	\w=(w,\theta)\in T^*Y$, and $\r = (r,\rho)$.  Also, the ``prime" operation is defined to be  $\y' := (y,-\eta)$.
	
	\item A canonical relation from a symplectic manifold $\calM_1$ to $\calM_2$ will
	be a Lagrangian submanifold of $\calM_2\times\calM_1^{-}$ (the domain of the relation is a subset of
	the second factor).
	
	\item If an FIO e.g. from $C_c^\infty(Y)$ to $C^\infty(X)$ has a Schwartz kernel $\calK \in C^{-\infty}(X\times Y)$,
	its canonical relation is
	\[
	\left\{(\x, \y)\;|\; (\x, \y')\in\wf(\calK)\right\}.
	\]
\end{itemize}

\subsection{The operators $T_\pm\psp(h^2\DY)$ and $R_\pm\psp(h^2\DY)$}
\label{s:mproperties}

\bigskip
In this section we prove that $T_+\psp(h^2\DY)$ and $R_-\psp(h^2\DY)$ are semi-classical
Fourier integral operators for any $\psp\in C_c^\infty([0,1))$.

\smallskip

\begin{proposition}
	Let
	\begin{equation}\label{}
		W_{\pm}(r,w,y) = \sum_j \psp(h^2\sigma_j^2) e^{\mp irh^{-1}\tau_j} \phi_j(w)\overline{\phi_j(y)},
	\end{equation}
	where $\psp\in C_c^\infty([0,1))$ and $\tau_j=\tau_j(h)=(1-h^2\sigma_j^2)^{1/2}$.  Then $W_{\pm}$ is a Lagrangian semi-classical
	function on $\bbR\times Y\times Y$, associated with the Lagrangian submanifold
	$\Gamma_{\pm}\subset T^*(\bbR\times Y\times Y)$ given by
	\begin{equation}\label{}
		\Gamma_{\pm}= \left\{
		(\r ; \w ; \y')\;|\; \y\in \calB,\, \w = \Psi_{\pm r}(\y),\, \rho  = \mp H(\y)
		\right\},
	\end{equation}
	where $\calB$ is the open unit tangent ball bundle of $Y$, 	$$H(y,\eta) = \sqrt{1-|\eta|^2}$$ and $\Psi$ the Hamilton flow of
	$H$.
\end{proposition}
\begin{proof}
	Microlocally in $\calB$, the operator $\sqrt{I-h^2\DY}$ is a semi-classical self-adjoint pseudodifferential operator
	of order zero.  The function $W_{\pm}$ is the Schwartz kernel of the composition 
	\[
	\psp(h^2\Delta)e^{\mp irh^{-1}\sqrt{I-h^2\DY}},
	\]
	regarded as an operator  $L^2(Y)\to L^2(\tilde X)$.
	It is well-known that if $A$ is a self adjoint semi-classical pseudodifferential operator of order zero, 
	the exponential $e^{-irh^{-1}A}$ is a semi-classical Fourier integral operator \cite[Theorem 11.5.1]{GS}, \cite[Section IV.6]{Rob}
	associated with a Lagrangian strictly analogous to $\Gamma_{+}$.  The presence
	of the factor $\psp(h^2\Delta)$ in $U$ microlocalizes $\sqrt{I-h^2\DY}$ to where it is a pseudodifferential
	operator, so the same construction can be applied verbatim to the $W_{\pm}$.  
\end{proof}

Note that the operator $T_\pm\psp(h^2\DY): L_c^2(\tilde{X})\to L^2(Y)$ has $W_{\pm}$ as
its Schwartz kernel, except for a trivial permuation of the variables:
\[
T_\pm\psp(h^2\DY) (f)(w) = \int W_{\pm}(r,w,y)\, f(r,y)\, dr\, dy.
\]
Similarly, $R_\pm\psp(h^2\DY)$, where $R_\pm$ is defined by (\ref{rInverse}), has for Schwartz kernel $W_{\mp}$, this time in 
the standard manner:
\[
R_\pm\psp(h^2\DY)(g)(r,w) = {\color{blue}}\chi(r)\int W_{\mp}(r,w, y)g(y)\, dy.
\]
Therefore:
\begin{corollary}\label{cor:FIOs}
	The operator $T_\pm\psp(h^2\DY): L_c^2(\tilde{X})\to L^2(Y)$ is a Fourier integral
	operator associated with the canonical relation
	\begin{equation}\label{eq:relationTpm}
		\Theta_{\pm} =		\left\{
		\big(\w \,;\,(\r;\y)\big)\;|\;  \y\in\calB,\, \w = \Psi_{\pm r}(\y),\, \rho  = \pm H(\y)
		\right\}.
	\end{equation}
	Moreover,  $R_\pm\psp(h^2\DY):L^2_c(Y)\to L^2(\tilde X)$ is a Fourier integral operator associated with
	the canonical relation 
	\begin{equation}\label{eq:relationRpm}
		\Theta_{\pm}^T =			\left\{
		\big((\r;  \w)\,;\, \y\big)|\; \y\in\calB,\, \w = \Psi_{\mp r}(\y),\, \rho  = \pm H(\w),\, r\in \mbox{supp}(\chi)
		\right\},
	\end{equation}
	which is the transpose of (\ref{eq:relationTpm}) (except for the restriction on $r$).
\end{corollary}

In what follows we will work with the compositions 
$e^{-i\tpsi P/h}R_-$ and $T_+ \tilde\chi(r) e^{-i\tpsi P/h}R_-$,
where $\tilde\chi$ is compactly supported on $X_\infty$.
We will use the composition theorem for FIOs to prove that each of these operators is an FIO,
\cite[Theorem 18.13.1]{GS}.
We will show that the clean-intersection hypothesis of that theorem is satisfied in each case.

\subsection{Geometric considerations}

\medskip
To better understand the previous canonical relations,
introduce the co-isotropic submanifold of $T^*X$
\[
\calC = p^{-1}(1) = \{\x=(x,\xi)\in T^*X\;|\; p(\x)=|\xi|^2 + V_0(x) = 1\},
\]
whose null-leaves are the trajectories of $\Phi$, the Hamilton flow of $p$.  
(We are working microlocally in a region of $T^*X$ where
$p^{-1}(1)$ is a submanifold, and therefore without loss of generality for simplicity we will assume it is a submanifold everywhere).
The null leaves of $\calC$ are the (unparametrized) Hamilton trajectories of $p$.
Let
\begin{equation}\label{}
	\calC_{\pm} := \left\{
	(\r ;\, \y)\in T^*X_\infty\;|\; \y\in\calB,\ \rho = \pm H(\y)
	\right\}
\end{equation}
where $\calB\subset T^*Y$ is the open unit cotangent bundle.
Note that $\calC_{\pm}\subset \calC$ and that $\rho\not=0$ on $\calC_{\pm}$.  Also introduce
the embeddings
\begin{equation}\label{}
	\nu_\pm: \calB \to\calC_{\pm},\qquad \nu_\pm(\y) = (0, \pm H(\y)\,;\, \y).
\end{equation}

\begin{proposition}
	The images $\calT_{\pm} := \nu_{\pm}(\calB) = \calC_\pm \cap\{r=0\}$ 
	are symplectic submanifolds, and $\nu_{\pm}: \calB\cong\calT_{\pm}$
	is a symplectomorphism.  Moreover, $\calT_{\pm}$ are 
	Poincaré cross sections of $\calC_{\pm}$.  Explicitly, 
	the null leaf of $\calC_{\pm}$ through $(\r;\, \y)$ intersects the transversal $\calT_{\pm}$ at 
	$\nu_{\pm}(\Psi_{\pm r}(\y))$.
\end{proposition}
\begin{proof}
	Let $G_t$ denote the Hamilton flow on $T^*Y$ of the square of the Riemannian norm function, $\gamma_Y(\y)= |\eta|^2$.
	Then on $T^*X_\infty$
	\[
	\Phi_{t}(r,\rho\,;\, \y) = (r+2t\rho, \rho\,;\, G_t(\y))
	\]
	as long as $r+2t\rho > -4$.
	On the other hand, $dH = -\frac{1}{2H}d\gamma_Y$, and therefore a similar relation holds among 
	the corresponding Hamilton fields of $H$ and $\gamma_Y$.  It follows that
	$G_t = \Psi_{-2tH}$ and on $T^*X_\infty$
	\begin{equation}\label{eq:PhiandPsi}
		\Phi_t(r,\rho\,;\,\y) = (r+2t\rho, \rho\, ;\,  \Psi_{-2tH(\y)}(\y)).	
	\end{equation}
	Therefore
	\begin{equation}\label{eq:onFlows}
		\text{for all}\;
		(r,\rho\,;\,\y)\in\calC_\pm\qquad	\Phi_{-r/2\rho} (r,\rho\,;\,\y) = \nu_\pm(\Psi_{\pm r}(\y)).
	\end{equation}
\end{proof}
Note that (\ref{eq:onFlows}) implies that for all $(r,\rho\,;\,\y)\in\calC_\pm$,  
$(r,\rho\,;\,\y) = \Phi_{r/2\rho}\left(\nu_\pm(\Psi_{\pm r}(\y))\right)$.
Replacing $\y$ by $\Psi_{\mp r}(\y)$ yields
\begin{equation}\label{}
	(r,\rho\,;\,\Psi_{\mp r}(\y)) = \Phi_{r/2\rho}\left(\nu_\pm(\y)\right).
\end{equation}
Therefore, we can re-state Corollary \ref{cor:FIOs} 
as follows: 
\begin{corollary}\label{cor:FIOs2}
	The canonical relation $\Theta_\pm^T$ of $R_\pm\psp(h^2\DY)$ is 
	\begin{equation}\label{}
		\Theta_\pm^T = \left\{ (\Phi_t(\nu_{\pm}(\y))\,,\, \y)\;|\; \y\in\calB,\ t = \pm\frac{r}{2H(\y)},\ r\in (-1/4,0)
		\right\}.
	\end{equation}
\end{corollary}

	\begin{remark}\label{rmk:changeInKappa}
We can use (\ref{eq:PhiandPsi}) to see how the scattering map changes if we change the origin of the $r$ coordinate.  
If we replace $r$ by $r-c$ for some constant $c>0$, we obtain a scattering map $\kappa'$ with domain $\calD_{\kappa '}$.  
A point $\y \in\calB$ is in the domain $\calD_{\kappa '}$ if and only if  the trajectory $t\mapsto \Phi_t(c, -H(\y); \y)$ is not forward-trapping.
By (\ref{eq:PhiandPsi}), 
\begin{equation}\label{eq:changeR=0}
\Phi_t(c, -H(\y); \y) = (c-2tH(\y), -H(\y)\,;\, \Psi_{-2tH(\y)}(\y)).
\end{equation}
Therefore this trajectory traverses the hypersurface $\{r=0\}$ at time $t = \frac{c}{2H(\y)}$ and at the point $(0, -H(\y); \Psi_{-c}(\y))$.
It follows that if we let
\[
\vartheta:= \Psi_{-c}: \calB\to\calB,
\]
then $\vartheta$ maps $\calD_{\kappa '}$ into $\calD_\kappa$.  The converse is analogous, that is, 
$\vartheta$ maps $\calD_{\kappa '}$ bijectively onto $\calD_\kappa$.

To find $\kappa'(\y)$, we are to follow the trajectory described above until the time $t'_+(\y)>0$ where it
intersects $\{r=c\}$, and then take the $T^*Y$ component of the point of intersection.  
By the previous discussion,
\[
t'_+(\y) = \frac{c}{2H(\y)} + t_{+}(\vartheta(\y)) + s
\]
where $s$ is such that
\[
\Phi_s(0, H(\kappa(\vartheta(\y))); \kappa(\vartheta(\y)))= (c,  H(\kappa(\vartheta(\y))); \kappa'(\y)).
\]
Applying (\ref{eq:changeR=0}) again to the left-hand side of this identity, we obtain
\[
\Phi_s(0, H(\kappa(\vartheta(\y))); \kappa(\vartheta(\y)))= (2sH(\kappa(\vartheta(\y))), H(\kappa(\vartheta(\y))); \Psi_{-2sH(\kappa(\vartheta(\y)))}(\kappa(\vartheta(\y))),
\]
that is, $s=c/2 H(\kappa(\vartheta(\y)))$.  Substituting, we have
\[
(c, H(\kappa(\vartheta(\y))); \Psi_{-c}(\kappa(\vartheta(\y))) = (c,  H(\kappa(\vartheta(\y))); \kappa'(\y)),
\]
which shows that $\kappa' = \vartheta\circ\kappa\circ\vartheta$.
\end{remark}

\subsection{The operator $T_+ \chi(r) e^{-i\tpsi P/h}R_-\psp(h^2\DY)\text{Op}_h(\psi)$}

Let us fix $\delta>0$ and use $\frakY$ (the fraktur letter $Y$) for 
\[
\frakY=\left\{ \y=(y,\eta)\in T^*Y\;|\;|\eta|<1-\delta
\right\}\cap \calD_\kappa 
\]
where $\delta>0$ is small enough that $\text{supp}(\psp)\subset [0, 1-\delta)$.
Let $\psi\in C_c^\infty(\frakY)$, and 
let $\tpsi >0$ be chosen sufficiently large so that if $(y,\eta)\in \supp \psi$ 
then $\Phi_t (0, -\sqrt{1-|\eta|^2}; y,\eta)\in \{(x,\xi)\in T^*\Xinf \; |\: \; r\geq 0\}$ for all $t\geq \tpsi$.  The existence of such a 
$\tpsi$  follows from the assumption that the support of $\psi$ is compact and
is contained in the domain of the scattering map. 

We first consider $e^{-i\tpsi P/h}R_-\psp(h^2\DY)\text{Op}_h(\psi)$.  

\begin{proposition}\label{prop:FIOI}
	For any $\tilde{\chi}\in C_c^\infty((-4,\infty))$ the composition $\tilde{\chi}(r)e^{-i\tpsi P/h}R_-\psp(h^2\DY)\text{Op}_h(\psi)$ is a Fourier integral operator whose
	canonical relation $\Sigma\subset T^*X \times T^*Y$ is
	\begin{equation}\label{wFSI}
		\Sigma = \left\{
		\left( \Phi_{t_\psi +t}(\nu_-(\y))\, ;\, \y \right)\in T^*X_\infty\times T^*Y,
		\;|\; \y\in\frakY,\ t\in\left(0, \frac{1}{8H(\y)}\right)
		\right\}.
	\end{equation}
\end{proposition}

\begin{remark}
	Clearly (\ref{wFSI}) is parametrized by
	\[
	\calD:=
	\left\{
	(\y, t)\in \frakY\times\bbR\; |\;  t\in\left(0, \frac{1}{8H(\y)}\right)
	\right\}.
	\]
	The condition on $t$ ensures that (\ref{wFSI}) is over the portion of the cylinder defined by $r\in (0,\infty)$.
\end{remark}
\begin{remark}
	It is important to note that by the assumption on $t_\psi$ and $\frakY$, in (\ref{wFSI})
	\[
	t\mapsto \Phi_{t_\psi +t}(\nu_-(\y_0))
	\]
	is an outward-going geodesic on $T^*X_\infty\cap\calC$.  It is therefore of the form
	\begin{equation}\label{}
		\Phi_{t_\psi +t}(\nu_-(\y_0)) = (r(t), H(\y(t))\,;\, \y(t))
	\end{equation}
	with $\frac{dr}{dt}>0$.
\end{remark}
\noindent
{\em Proof of Proposition \ref{prop:FIOI}.}  Again by \cite[Theorem 11.5.1]{GS}, the factor $e^{-i\tpsi P/h}$
is a Fourier integral operator associated to the graph of $\Phi_{\tpsi}$.  It is known (\cite[\S4.3]{GS})
that left-composition by an FIO associated with a canonical {transformation} is always clean (in fact, transverse), 
and therefore  $\chi(r)e^{-i\tpsi P/h}R_-$ is a Fourier integral operator whose canonical 
relation is the composition of the graph of $\Phi_{\tpsi}$ with $\Theta_-^T$. 
The result follows directly from Corollary \ref{cor:FIOs2}.

\hfill $\Box$.

\medskip 
For $\tilde{\chi}\in C_c^\infty((-4,\infty))$ and $\psp\in C_c^\infty(\mathcal{B})$,  we analyze the composition
\begin{equation}\label{eq:theCompo}
	\psp(h^2\DY) T_+ \circ \left(\tilde{\chi}(r) e^{-i\tpsi P/h}R_-\psp(h^2\DY)\text{Op}_h(\psi),\right),
\end{equation}
which is a bit more complicated.

\begin{proposition}\label{p:mainML}
	The operator (\ref{eq:theCompo}) is an FIO, associated to the graph of the scattering map $\kappa$ restricted to $\mathfrak{Y}$.
\end{proposition}
\begin{proof}  
	Introduce the manifolds:
	\[
	T^*X\buildrel \Delta \over \times T^*X := \left\{ (\x ; \x)\;|\; \x\in T^*X)
	\right\},
	\quad
	\mathfrak A = T^*Y\times \left(T^*X\buildrel \Delta \over \times T^*X\right)\times T^*Y,
	\]
	and $\mathfrak B = \Theta_+ \times \Sigma$.  Recall that $\Theta_+$ and $\Sigma$ are the canonical relations associated
	to the factors of the composition (\ref{eq:theCompo}), and note that $\mathfrak A$ and $\mathfrak B$
	are submanifolds of $T^*Y\times T^*X\times T^*X\times T^*Y$.

	We first prove that
		the manifolds $\mathfrak A$ and $\mathfrak B$ intersect cleanly.
	We claim that the intersection is the set
	\begin{equation}\label{eq:theIntersection}
		\mathfrak A\cap\mathfrak B = \left\{
		\left( \kappa(\y)\,;\, \Phi_{t_\psi +t}(\nu_-(\y))\,;\,\Phi_{t_\psi +t}(\nu_-(\y))\,;\, \y\right)\;|\; (\y, t) \in\calD 
		\right\},
	\end{equation}
	where $\kappa$ is the scattering map.  
	To see this, let $\zeta\in \mathfrak A \cap \mathfrak B$, and
	let us write 
	\[
	\zeta = (\w\,;\, \Phi_s(\nu_+(\w))\,;\, \Phi_{\tpsi +t}(\nu_-(\y))\,;\,\y),
	\]
	where $\Phi_s(\nu_+(\w))= \Phi_{\tpsi +t}(\nu_-(\y))$.   Therefore
	\[
	\nu_+(\w) = \Phi_{\tpsi +t-s}(\nu_-(\y)),
	\]
	which is a relation that characterizes the scattering map, namely
	\begin{equation}\label{eq:eqForKappa}
		\nu_+(\kappa(\y)) = \Phi_{t_+(\y)}(\nu_-(\y)),
	\end{equation}
	where $\y\to t_+(\y)$ is a smooth function by the implicit function theorem.
	Therefore $	\w=\kappa(\y)$,	which yields (\ref{eq:theIntersection}).
	We also obtain the relation $t_+(\y) = \tpsi +t-s$.
	
	The set $\mathfrak A\cap\mathfrak B$ is clearly a submanifold parametrized by 
	$(\y, t)\in\calD$, and elements in $T_\zeta(\mathfrak A\cap\mathfrak B)$ are of the form
	\begin{equation}\label{eq:tangentVectors}
		\left( d\kappa(\delta\y)\,;\,d\Phi_{\tpsi+t}\circ d\nu_-(\delta\y)+\delta t\Xi
		\,;\,d\Phi_{\tpsi+t}\circ d\nu_-(\delta\y)+\delta t\Xi\,;\,\delta\y
		\right)
	\end{equation}
where $\delta\y\in T_{\y}T^*Y$, $\delta t\in T_t\bbR \cong\bbR$,  
and $\Xi$ is the Hamilton field of $p$ (the generator of $\Phi$) evaluated at the appropriate point.
	
	To prove that the intersection is clean, we need to show that
	\[
	T_\zeta\left(\mathfrak A\cap\mathfrak B \right) = T_\zeta\mathfrak A\cap T_\zeta\mathfrak B.
	\]
	The inclusion $T_\zeta\left(\mathfrak A\cap\mathfrak B \right) \subset T_\zeta\mathfrak A\cap T_\zeta\mathfrak B$ is automatic, 
	so let $v\in T_\zeta\mathfrak A\cap T_\zeta\mathfrak B$.
	Since $v\in T_\zeta\mathfrak B$, it is of the form
	\[
	v = \big(\delta\w\,;\, d(\Phi_s)\circ d\nu_+(\delta\w) + \delta s\Xi\,;\, 
	d(\Phi_{\tpsi +t})\circ d\nu_-(\delta \y)+\delta t\Xi\,;\, \delta\y\big)
	\]
	where $\delta\w \in T_{\w}\calB$, $\delta s\in T_s\bbR\cong\bbR$, etc.,
	That $v\in T_\zeta\mathfrak A$ means that the middle entries in $v$ are equal, that is
	\begin{equation}\label{eq:dStuff}
		d(\Phi_s)\circ d\nu_+(\delta\w) + \delta s\Xi\ =
		d(\Phi_{\tpsi +t})\circ d\nu_-(\delta \y)+\delta t\Xi.
	\end{equation}
	Comparing with (\ref{eq:tangentVectors}), in order to 
	conclude that $v\in T_\zeta\left(\mathfrak A\cap\mathfrak B \right)$ all we need to show is that
	$\delta\w = d\kappa(\delta\y)$.  To see this, let us rewrite (\ref{eq:dStuff}) as
	\[
	d(\Phi_{\tpsi +t-s})\circ d\nu_-(\delta \y) = d\nu_+(\delta\w) +(\delta s-\delta t)\Xi.
	\]
	But by (\ref{eq:eqForKappa}), this also equals
	\[
	d(\Phi_{t_+(\y)})\circ d\nu_-(\delta \y) = d\nu_+(d\kappa(\delta\y)) - dt(\delta\y)\Xi.
	\]
	Now the summands on the right-hand sides of these expressions correspond to the direct sum decomposition 
	$T\calC_+=T\calT_+\oplus \bbR\Xi$.
	Therefore, corresponding summands must equal each other, that is
	\[
	d\nu_+(\delta\w) = d\nu_+(d\kappa(\delta\y))\quad\text{and}\quad (\delta s-\delta t)\Xi =-dt(\delta\y)\Xi.
	\]
	Since $d\nu_+$ is injective, the first of these relations yields $\delta\w = d\kappa(\delta\y)$, and the proof
	that $\mathfrak A$ and $\mathfrak B$ intersect cleanly	is complete.
	
	By the composition theorem for semi-classical FIOs \cite[Theorem 18.13.1]{GS}, the operator (\ref{eq:theCompo}) is a semi-classical FIO
	associated to the relation which is the image of $\mathfrak{A}\cap\mathfrak{B}$ under the projection onto $T^*Y\times T^*Y$.
	By (\ref{eq:theIntersection}), this is precisely the graph of $\kappa$ restricted to $\mathfrak{Y}$.
\end{proof}

\section{A microlocal approximation of  the Poisson operator and  the scattering matrix }
\label{s:gef}

  In this section we give a microlocal construction of $\pop \Oph(\psi)$, the Poisson operator composed with $\Oph(\psi)$.  
  Recall $\psi\in C_c^\infty(T^*Y)$ has support contained in the domain of the scattering map $\kappa$. 
A consequence of our construction is an expression for the scattering matrix in terms of $R_-$, $T_+$, and the Schr\"odinger propagator, see Proposition \ref{p:smident}.
Propositions \ref{p:mainML} and \ref{p:smident} combine to  prove our theorems.

Recall $\tpsi >0$ is chosen sufficiently large so that if $(y,\eta)\in \supp \psi$, then $\Phi_t ( 0, -\sqrt{1-|\eta|^2}; y,\eta)\in \{(x,\xi)\in T^*\Xinf: \; r\geq 0\}$ for all $t\geq \tpsi$.   Here we continue to use the notation for the cotangent variables on $T^*\Xinf$ introduced in Section \ref{s:mp}.
Choose $M\in \mathbbold{N}$ so that 
if $(y,\eta)\in \supp \psi$ 
and $-1/4\leq s\leq 0$ then 

	\begin{equation}\label{eq:condition}
	\forall \; t\in[0,\tpsi]\qquad
\Phi_{t}(s, -\sqrt{1-|\eta|^2};\Psi_s( y,\eta))\in \{(x,\xi)\in T^*\Xinf: \; r\leq M-2\}.
\end{equation}

In particular, this implies $M\geq 2$. Let $b_\psi<1$ be chosen so that if $(y,\eta)\in \supp (\psi)$ or $(r, \sqrt{1-|\eta|^2};y,\eta)=\Phi_{\tpsi}(r_-, -\sqrt{1-|\eta_-|^2}; y_-,\eta_-)$
for some $(y_-,\eta_-)\in \supp \psi$ and $-1/4\leq r_-\leq 0$, then $|\eta|\leq b_\psi$.  
 Choose $\psp\in C_c^\infty([0,1))$ so that $\psp$ is $1$ on $[0,b_{\psi}]$.  

 Let $\chi_j\in C^\infty(\R;[0,1]) $ satisfy $\chi_j(r)=1$ if $r<-2+j$ and 
$\chi_j(r)=0$ if $r>-3/2+j$, ensuring $\chi_j\chi_{j+1}=\chi_j$. 

Recall $R_-$ is defined in (\ref{rInverse}), and $\supp(R_-f_-)\subset \{(r,y)\in \Xinf \mid -1/4\leq r \leq 0\}$.  
Set 
\begin{align}
U_-& = (I-h^2\DY)^{1/2}\psp(h^2\DY)R_- \Oph(\psi):L^2(Y)\rightarrow H^\infty(\Xinf)\subset H^\infty(X), \nonumber \\
 U_+&= (\chi_{M}(r)-\chi_1(r))e^{-i\tpsi P/h}U_-:L^2(Y)\rightarrow L^2(\Xinf)\subset L^2(X)
 \end{align}
and 
\begin{multline}\label{eq:popapp} \popapp = 2ih\left( (1-\chi_0(r))
(h^2\DXt -1+i0)^{-1}U_-
+
\frac{1}{ih} \chi_M(r)\int_0^{\tpsi}e^{it/h}e^{-it\Ph/h} U_- dt  \right. \\ \left.
-e^{i\tpsi /h} (1-\chi_0(r))(h^2\DXt -1-i0)^{-1}\left(\psp(h^2\DY)U_+ \right)\right).
\end{multline}
We shall see that the operator $\popapp$ is an approximation of $\pop \Oph(\psi)$.
The mapping properties of $(h^2\DXt -1\pm i0)^{-1}$  ensure that if $f_-\in L^2(Y)$, then $\popapp  f_-\in \langle r \rangle ^{1/2+\delta}H^\infty(X)$ for any $\delta >0$.   Note
that our definition of $\popapp$ involves $\Oph(\psi)$, and so depends on choice of $\psi$, even though our notation does not indicate this.

We begin with a preliminary lemma.	
\begin{lemma}\label{l:uniform}
Let $\tilde{\chi}\in C_c^\infty(\Xinf)$  have support in the region with $r\geq M-2\geq 0$.  
Then $\| \tilde{\chi} e^{-it\Ph/h} U_- \|=O(h^\infty)$  uniformly for $t$ satisfying $0\leq t \leq  \tpsi$.  
\end{lemma}
\begin{proof}
First observe that, by (\ref{rInverse}),
\[
\forall \;g\in L^2(Y)\qquad U_{-}g= \chi(r )(I-h^2\DY)^{1/2}\psp(h^2\DY)e^{- ir(I-h^2\DY)^{1/2}_+/h}g,
\]
where $\chi$ is supported in $(-1/4,0)$.
	
For each value of $h$ and each $t$, the operator $e^{-it\Ph/h} U_-$ 
is a smoothing operator of finite rank.  By 
Corollary \ref{cor:FIOs} and the composition theorem for FIOs, it is also
a semi-classical FIO whose canonical relation is 
	\begin{equation}\label{eq:relationLemma6.1}
\left\{
	\big(\Phi_t(r,\rho;  \w)\,,\, \y\big)|\; \w = \Psi_{r}(\y),\, \y\in\mbox{supp}(\psi),\, \rho  = - H(\w),\, r\in \mbox{supp}(\chi)
	\right\}.
\end{equation}

More precisely, let $\chi^\sharp\in C_c^\infty(\bbR)$ be identically equal to one in a neighborhood of $1\in\bbR$, 
and note that
\[
\chi^\sharp(h^2\DX)U_- = U_- + O(h^\infty)
\]
because the image of the canonical relation of $U_-$ (which is the same as that of $R_-$) is contained in $p^{-1}(1)$.  
We now use the well-known approximation of $e^{-itP/h}\chi^\sharp(h^2\DX)$ by 
oscillatory integrals, uniformly for $t$ in a bounded interval, see e.g. \cite[Theorem IV-30]{Rob} or \cite[Lemma 3.2]{ABR}.
That is, one can write
$e^{-itP/h}\chi^\sharp(h^2\DX) = \calF_1 +\calR_1,$
where the Schwartz kernel of $\calF_1$ is a finite sum of oscillatory integrals of the form 
$\int e^{ih^{-1}\varpi(t,x,x',p)} a(t,x,x',p,h) dp$
where $\varpi$ are generating functions for portions of the canonical relation of the graph of
the Hamilton flow of $p$, the amplitudes
$a$ are smooth and have an asymptotic expansion in powers of $h$, and the Schwartz kernel of
$\calR_1$ is $O(h^\infty)$ uniformly for $t$ in a bounded interval.  
Similarly, one can write $U_- = \calF_2 + \calR_2$,
where the Schwartz kernel of $\calF_2$ is a finite sum of oscillatory integrals of the form 
$\int e^{ih^{-1}\varpi(y,x,p)} a(y,x,p,h) dp$
where $\varpi$ are generating functions for the canonical relation of $R_-$, and the Schwartz kernel of
$\calR_2$ is $O(h^\infty)$.  
It follows that
\[
e^{-itP/h}\chi^\sharp(h^2\DX)U_- = \calF_1\calF_2 + \calS,\qquad \calS = \calF_1\calR_2+\calR_1\calF_2+\calR_1\calR_2.
\]
Note that the Schwartz kernel of $\calS$ is $O(h^\infty)$ uniformly for $t$ in a compact interval.

Now recall how 
$M\in \mathbbold{N}$ is chosen, (\ref{eq:condition}), and 
also recall that $\text{supp}(\chi)\subset (-1/4,0)$.
It follows that the Schwartz kernel of 
$\tilde{\chi}\calF_1\calF_2$ is a finite sum of oscillatory integrals whose phase functions do not
have critical points in the support of their amplitudes.
Therefore the Hilbert-Schmidt norm of $\tilde{\chi}\calF_1\calF_2$
can be estimated as $h\to 0$ by a finite sum of absolute values of  oscillatory 
integrals without critical points.  By smoothness of the integrands, the estimate is 
uniform in $t\in [0, t_\psi]$.

In combination with the rapid decrease of the Schwartz kernel of $\calS$, 
we can conclude that the Hilbert-Schmidt norm of $\tilde{\chi}e^{-itP/h}U_-$
is $O(h^\infty)$ uniformly in $t\in [0, t_\psi]$.
\end{proof}

\begin{lemma}\label{l:Euhm} Set $(\Ph-1)\popapp=2ih E$.  Then for any $f_-\in L^2(Y)$, 
$E f_-$ is compactly supported with support 
in $X_c\cup \{ (r,y)\in \Xinf: r\leq M-1 \}$, and $\| E \|_{L^2(Y)\rightarrow H^2(X)}=O(h^\infty)$.
\end{lemma}
\begin{proof}
Using that
$$(\Ph-1)\int_0^{\tpsi}e^{it/h}e^{-it\Ph/h}  dt
= ih\left(e^{i\tpsi /h}e^{-i\tpsi \Ph/h} - I\right)
$$
and $ (1-\chi_0)U_-=U_-=\chi_M U_-$
gives
$E= \sum_{j=1}^4E_j,$
where 
\begin{align*}
E_{1}& =[h^2\partial_r^2,\chi_0] (h^2\DXt -1+i0)^{-1} U_- \\
E_{2}& = -\frac{1}{ih} [h^2\partial^2_r,\chi_M]\int_0^{\tpsi }e^{it/h}e^{-it\Ph/h} U_- dt\\
E_{3}& = e^{i\tpsi /h}\left(  \chi_M e^{-i\tpsi \Ph/h} U_-  -\psp(h^2\DY)U_+\right) \\
E_{4} & = -e^{i\tpsi /h}[h^2\partial_r^2,\chi_0](h^2\DXt -1-i0)^{-1}
\psp(h^2\DY)U_+.
\end{align*}
The claim about the support of $E$ is immediate from our expression for
$E$.  

We begin with bounding $E_2$.   
Since $[h^2\partial^2_r,\chi_M]$ is supported in $\{x\in \Xinf \mid  x=(r,y), M-2\leq r \leq -3/2+M\}$, as a corollary of Lemma \ref{l:uniform} we 
obtain that $\| E_2\|=O(h^\infty)$.

For $E_3$, we use 
\begin{multline}
\|E_3\| = \|\chi_M e^{-i\tpsi \Ph/h} U_-  -\psp(h^2\DY)(\chi_M-\chi_0)e^{-i\tpsi \Ph/h} U_-\|
\\
\leq \|(1-\psp(h^2\DY))(\chi_M-\chi_0)e^{-i\tpsi \Ph/h} U_- \| + \| \chi_0 e^{-i\tpsi \Ph/h} U_- \|.
\end{multline}
That $\| (1-\psp(h^2\DY))(\chi_M-\chi_0)e^{-i\tpsi \Ph/h} U_- \|=O(h^\infty)$ follows from Proposition \ref{prop:FIOI} and our choice of $\psp$, and that $\| \chi_0 e^{-i\tpsi \Ph/h} U_- \|=O(h^\infty)$ follows
from Proposition \ref{prop:FIOI}, the support properties of $\chi_0$, and our choice of $t_\psi$.

Now consider $E_4$.  The support properties of $[h^2 \partial^2_r,\chi_0]$ and $(\chi_M-\chi_1)$ mean that by Lemma \ref{l:prodresolve}
$$\| E_4\|= (2h)^{-1}\| (h^2 \chi_0'' (1-h^2\DY)^{-1/2}- 2i\chi_0'  )\psp(h^2 \DY) T_ - U_+\|.$$
But by Corollary \ref{cor:FIOs}
 and Proposition \ref{prop:FIOI},  the composition of the canonical relations of $T_-$ and $U_+$ is empty.  Therefore
$\| T_- U_+\|_{L^2(Y)\rightarrow H^m(Y)}=O(h^{\infty})$, and hence $\| E_4\|=O(h^{\infty})$.

The term $E_1$ is handled in a way similar to $E_4$, using that 
 $$(h^2\DXt -1+i0)^{-1}f = \overline{(h^2 \DXt-1-i0)^{-1}\overline{f}}.$$
\end{proof}

For the next lemma, we continue to use the functions $\chi_j$ introduced above.
\begin{lemma}\label{l:endbds}
Suppose $f\in L^2_c(X)$ has support in $X_C\cup\{(r,y)\in \Xinf \mid -4<r\leq M'\}$, with $-4<M'<\infty$.
Then for $r\geq M'$,
\begin{multline}\label{eq:morerestrict}
\| \left( \mathbbold{1}_{[0,1-\epsilon]}(h^2\DY)(1-\chi_0)(\Ph-1-i0)^{-1} f\right) (r, \bullet)\|_{L^2(Y)}\\
\leq 
\|\mathbbold{1}_{[0,1-\epsilon]}(h^2\DY)\mathbbold{1}_{[M',M'+1]}(r)(\Ph-1-i0)^{-1}\mathbbold{1}_{(-\infty,M']}(r)\|
\|f\|
\end{multline}
for any $\epsilon>0$.
Moreover, 
\begin{multline}\label{eq:lessrestrict}
\|\left( \mathbbold{1}_{[0,1)}(h^2\DY)(1-\chi_0)(\Ph-1-i0)^{-1} f\right) (r, \bullet)\|_{L^2(Y)} \\
\leq 
\|\mathbbold{1}_{[0,1)}(h^2\DY)\mathbbold{1}_{[M',M'+1]}(r)(\Ph-1-i0)^{-1}\mathbbold{1}_{(-\infty,M']}(r)\|
\|f\|.
\end{multline}
\end{lemma}

Although these two are almost the 
same, and have essentially identical proofs, the 
operator norm on the right side of  (\ref{eq:morerestrict}) may be smaller
than that in (\ref{eq:lessrestrict}); see, for example, Section \ref{ss:bulge}.  
\begin{proof}
From the support properties of $f$,  there are  $c_j=c_j(h,f)\in \C$ so that for $r>M'$
$$(\mathbbold{1}_{[0,1-\epsilon]}(h^2\DY)(1-\chi_0)(\Ph-1-i0)^{-1} f )(r, y)=\sum_{h^2\sigma_j^2\leq 1-\epsilon} c_j e^{i\tau_j r/h} \phi_j(y)$$
where $\tau_j=(1-\sigma^2_jh^2)^{1/2}\geq 0$ for $h^2\sigma_j^2\leq 1$.
Then for such $r$
\begin{align*} & 
\|\left( \mathbbold{1}_{[0,1-\epsilon]}(h^2\DY)(1-\chi_0)(\Ph-1-i0)^{-1} f\right) (r, \bullet)\|_{L^2(Y)} ^2\\& = \sum_{h^2\sigma_j^2\leq 1-\epsilon}|c_j|^2\\ &
= \int_{\Xinf \; |\: M'\leq r \leq M'+1}|  \mathbbold{1}_{[0,1-\epsilon]}(h^2\DY)(1-\chi_0)(\Ph-1-i0)^{-1} f|^2
\\ & =  \|\mathbbold{1}_{[0,1-\epsilon]} \mathbbold{1}_{[M',M'+1]}(r)( \Ph-1-i0)^{-1}\mathbbold{1}_{(-\infty,M']}(r)f\|^2
\end{align*}
proving (\ref{eq:morerestrict}).
The proof of (\ref{eq:lessrestrict}) is essentially identical.
\end{proof}




Recall that $\popapp$ as defined in (\ref{eq:popapp}) depends on  $\psi \in C^\infty(T^*Y)$.
Set 
\begin{equation}\label{eq:qdef}
Q= \popapp - (P-1-i0)^{-1} (P-1)\popapp.
\end{equation}
Then $(P-1)Q =0$, and for any $f_-\in L^2(Y)$, $\delta >0$, $\langle r \rangle ^{-1/2-\delta}Qf_-\in L^2(X)$.   We shall show that $Q$ is actually $\pop {\psp}(h^2\DY) \Oph(\psi)$, and that under the 
hypotheses of Theorem \ref{thm:v1} or \ref{thm:v2} we can use this to find an expression for the cut-off scattering matrix, up to a small error.

\begin{proposition}\label{p:smident}
If $1\not \in \spec(h^2\DY)$, then the operator $Q$ defined in (\ref{eq:qdef}) is $Q=\pop \psp(h^2\DY)\Oph(\psi)$.  Moreover, if (\ref{eq:rbdhyp}) holds, then
\begin{multline} \label{eq:smfirst}
S \Oph(\psi)\\= e^{i\tpsi/h}(I-h^2\DY)^{-1/2} \psp(h^2\DY)T_+ (\chi_M-\chi_1) e^{-i\tpsi P/h} (I-h^2\DY)^{1/2} R_-\psp(h^2\DY)\Oph(\psi)+O(h^\infty),
\end{multline}
and if the hypotheses of Theorem \ref{thm:v2}  hold, 
\begin{multline} \label{eq:smsecond}
\tS \mathbbold{1}_{[0,1-\epsilon]}(h^2\DY)\Oph(\psi)\\
=  e^{i\tpsi/h}(I-h^2\DY)^{-1/2} \mathbbold{1}_{[0,1-\epsilon]}(h^2\DY)T_+ (\chi_M-\chi_1) e^{-i\tpsi P/h} (I-h^2\DY)^{1/2}R_- \psp(h^2\DY)\Oph(\psi)+O(h^\infty)   
\end{multline}
for any $\epsilon>0$.
If $1\in \spec (h^2 \DY)$ the equations (\ref{eq:smfirst}) and (\ref{eq:smsecond}) hold as limits as $h' \uparrow h$.
\end{proposition}
\begin{proof}
As noted already, $(\Ph-1)Q=0$.  Thus to show that $Q=\pop \Oph(\psi)$ it remains to study the expansion of $Q f_-$ on $\Xinf$ for
$f_-\in L^2(X)$.  We begin by studying the behavior of  $\popapp f_-$ for $r>M$.  By Lemma \ref{l:prodresolve}, 
\begin{align}& (1-\chi_0(r))(h^2\DXt -1+i0)^{-1} U_- f_- \rest_{r>M}\nonumber \\
& =\frac{-i}{2h} e^{-ir\sro/h}(I-h^2\DY)^{-1/2}T_-  \psp(h^2\DY)\sro R_- \Oph(\psi)f_- \rest_{r>M}\nonumber \\
&= \frac{-i}{2h} e^{-ir\sro/h}\psp(h^2\DY)\Oph(\psi) f_-\rest_{r>M}
\end{align}
using that $T_\pm$ commutes with functions of $h^2\DY$ and $T_\pm R_\pm=I$.  Also by Lemma \ref{l:prodresolve}, 
\begin{multline}
(1-\chi_0(r))(h^2\DXt -1-i0)^{-1}\left(\psp(h^2\DY)U_+ f_-\right)\rest_{r>M}\\=
\frac{i}{2h}e^{ir\sro/h}(1-h^2\DY)^{-1/2} T_+\left( \psp(h^2\DY)U_+ f_-\right)\rest_{r>M}.
\end{multline}
The term $\chi_M (r)\int_0^{\tpsi }e^{it/h}e^{-it\Ph/h} U_- dt$ in (\ref{eq:popapp})
vanishes if $r>M$. 

If $1\not \in \spec(h^2\DY)$,
\begin{equation}\label{eq:rescont}
\left( (1-\chi_0)(\Ph-1-i0)^{-1}E f_- \right)(r,y)\rest_{r>M}=\sum _{0\leq h^2\sigma_j^2}c_j e^{i\tau_j r/h}\phi_j(y)
\end{equation}
for some  $c_j\in \C$.   Here $\tau_j=(1-h^2\sigma_j^2)^{1/2}$ has $\Re \tau_j\geq 0$, $\Im \tau_j\geq 0$.
Combining these four observations, we see that if $1\not \in \spec(h^2\DY)$,
\begin{align}\label{eq:Qfexp}
(Qf_-)\rest_{r>M} =&\, e^{-ir\sro/h}\psp(h^2\DY)\Oph(\psi)f_-  \nonumber \\ & + e^{ir\sro/h} e^{i\tpsi/h}(I-h^2\DY)^{-1/2}T_+\left( \psp(h^2\DY)U_+ f_-\right)\nonumber \\
& -2ih \sum _{0\leq h^2\sigma_j^2}c_j e^{i\tau_j r/h}\phi_j(y).
\end{align}
This shows that $Q= \pop \psp(h^2\DY)\Oph(\psi)$.

We now turn to proving (\ref{eq:smfirst}), so  suppose (\ref{eq:rbdhyp}) holds.
Then using in addition Lemmas \ref{l:reX1} and \ref{l:reX2}, for any $M'>0$ there is a constant $C$ so that 
$$\| \mathbbold{1}_{[M', M'+1]}(r)(\Ph-1-i0)^{-1}\mathbbold{1}_{(-\infty,M']}(r)\|\leq  C(h^{-2} + h^{-N_0})\; \text{for}\; 0<h\leq h_0.$$
Then  by (\ref{eq:rescont}) and Lemmas \ref{l:Euhm} and \ref{l:endbds}, 
$ \sum _{0\leq h^2\sigma_j^2\leq 1}|c_j|^2=O(h^\infty\|f_-\|)$ where the $c_j$ are defined via (\ref{eq:Qfexp}).  Thus by our definition of the scattering matrix and (\ref{eq:Qfexp})
$$S\psp(h^2\DY)\Oph(\psi) f_-= e^{i\tpsi/h}(1-h^2\DY)^{-1/2}T_+\left( \psp(h^2\DY)U_+ f_-\right) +O(h^\infty \|f_-\|).$$
Using $\|(1-\psp(h^2\DY))\Oph(\psi) \|=O(h^\infty)$  and the fact that Lemma \ref{l:Sest}  and (\ref{eq:rbdhyp}) imply $\|\tS\|=O(h^{1-\max(2,N_0)})$ finishes the proof when (\ref{eq:rbdhyp}) holds,
if $1\not \in \spec(h^2\DY)$.  If $1\in \spec(h^2\DY)$, then the equality holds by taking the limits as $h'\uparrow h$.

Now suppose the hypotheses of Theorem \ref{thm:v2} hold.  Since $P$ commutes with $\Delta_{Y_0}$, the scattering matrix $S$ commutes with $\Delta_{Y_0}$
and, as mentioned earlier, this implies $\|\tS\|=1.$   Thus 
$$\| \tS (1-\psp(h^2 \DY))\Oph(\psi)\|= \| (1-\psp(h^2 \DY))\Oph(\psi)\| =O(h^\infty).$$  Applying Lemmas \ref{l:reX1}, \ref{l:reX2}, \ref{l:Euhm},  and \ref{l:endbds} 
as before gives
$ \sum _{0\leq  h^2\sigma_j^2\leq 1-\epsilon }|c_j|^2=O(h^\infty)$, where the $c_j$ are as in (\ref{eq:Qfexp}).
Thus (\ref{eq:smsecond}) holds.
\end{proof}

\vspace{2mm}
\noindent{\em Proof of Theorems \ref{thm:v1} and \ref{thm:v2}.} 
Combining Propositions \ref{p:mainML} and  \ref{p:smident} proves Theorems \ref{thm:v1} and \ref{thm:v2} for $S\Oph(\psi)$.

Turning to the proof for  $\SU\Oph(\psi)$,  choose $\psp \in C_c^\infty([0,1))$ so that 
$$(I-\psp(h^2\DY))\Oph(\psi_0)=O(h^\infty).$$
Then the unitarity of $\SU$ implies
 $$\SU\Oph(\psi)= (I-h^2\DY)_+^{1/4}S (I-h^2\DY)_+^{-1/4}\psp(h^2\DY)\Oph(\psi)+O(h^\infty).$$
  Since $(I-h^2\DY)_+^{-1/4}\psp(h^2\DY)\Oph(\psi)$ is a pseudodifferential operator with symbol supported in the support of 
 $\psi$, using the result for $S$ we see there is a $\psi_1\in C_c^\infty([0,1))$ so that
$$\|(I-\psi_1(h^2\DY))S\psp(h^2\DY)(I-h^2\DY)_+^{-1/4}\psp(h^2\DY)\Oph(\psi)\|=O(h^\infty).$$
Thus 
$$\SU\Oph(\psi)= (I-h^2\DY)_+^{1/4}\psi_1(h^2\DY)S (I-h^2\DY)_+^{-1/4}\psp(h^2\DY)\Oph(\psi)+O(h^\infty),$$
and the result for $\SU$ follows from the result for $S$ and composition properties of Fourier integral and pseudodiffierential operators.
\qed

\section{Equidistribution of phase shifts} \label{s:eps}
As an application of our theorems on the microlocal structure of the unitary scattering matrix $S_U$, 
in this section we prove Theorem \ref{thm:equidist}, a result about the distribution of its phase shifts.  This requires some
additional hypotheses, for which we need some background.

\subsection{Distance on $T^*Y$ and Minkowski content}
Fix any smooth Riemannian metric on $T^*Y$.  This induces a distance on each connected component of $T^*Y$.  
If $\overline{y}, \;\overline{w}\in T^*Y$ belong to different connected components of $T^*Y$, we shall say the (generalized)  distance (in $T^*Y$) between them is infinite.
We will denote this (generalized) distance by  $\operatorname{dist}_{T^*Y}$; $\operatorname{dist}_{T^*Y}:T^*Y\times T^*Y \rightarrow [0,\infty]$.
We use this to define the $(2n-2)$-dimensional Minkowski content of a bounded set $A\subset T^*Y$, where $2n-2=\dim(T^*Y)$.
The $(2n-2)$-dimensional upper Minkowski content of $A$  is
$$\mathcal{M}^{*2n-2}(A)=\lim \sup_{\delta \downarrow 0}  \mu( \{\overline{y}\in T^*Y \mid  \operatorname{dist}_{T^*Y}(\overline{y},A)<\delta\})$$
where for $B\subset T^*Y$, $\mu(B)$ is the Liouville measure of $B$.  Similarly, the $(2n-2)$-dimensional lower Minkowski content is
$$\mathcal{M}_*^{2n-2}(A)=\lim \inf_{\delta \downarrow 0}  \mu( \{\overline{y}\in T^*Y \mid  \operatorname{dist}_{T^*Y}(\overline{y},A)<\delta\}).$$
If $\mathcal{M}^{* 2n-2}(A)= \mathcal{M}_*^{2n-2}(A)$, then the $(2n-2)$-dimensional Minkowski content of $A$ is $\mathcal{M}^{ 2n-2}(A)= \mathcal{M}_*^{2n-2}(A)$.

For general $A$, the Minkowski content may depend on the choice of the metric on $T^*Y$ via the induced distance or the chosen measure.  
However, we shall only apply this for bounded sets $A$ that have zero 
$(2n-2)$-dimensional Minkowski content.  For such sets, the property of having zero 
Minkowski content is independent of the choice of  smooth metric on $T^*Y$.  Moreover, this is also true of the choice of measure, as long as the measures are mutually
absolutely continuous.

\begin{remark}  A set in a $d$-dimensional manifold that has zero $d$-dimensional Minkowski content
	has measure zero, but the converse is not true.  For example, let $\calQ$ be the intersection of the
	unit cube in $\bbR^d$ with $\bbQ^d$.  Then $\calQ$ has measure zero but  $d$-dimensional Minkowski 
	content one.
\end{remark}

%
%
%

\subsection{Hypotheses and Theorem \ref{thm:equidist}}\label{ss:edst}
Throughout Sections \ref{ss:edst} and \ref{ss:edp}, we assume:
\begin{enumerate}
\item The assumptions of at least one of Theorems \ref{thm:v1} and \ref{thm:v2} hold.
\item \label{as:Mink}For $m\in \Z$, let $\mathcal{D}_{\kappa^m}\subset \mcB$ be the domain of $\kappa^m$, where we recall $\kappa$ is the 
scattering map.  We assume that for each $m\in \Natural$ the $(2n-2)$-dimensional  Minkowski content of 
$\mcB\setminus \mathcal{D}_{\kappa^m}$ is $0$.  
\item \label{as:fp} For each $m\in \Z \setminus\{0\}$, the set of fixed points of $\kappa^m$ has measure $0$.
\end{enumerate}
In reference \cite{GHZ}, where the authors studied the equidistribution property for semiclassical Schr\"odinger operators on $\R^n$, the analogs of the first and second assumptions are implied by 
a non-trapping assumption, while the analog of the third assumption is made explicitly.  The proof we give here follows in outline much of 
the strategy of \cite{GHZ}.  
Some differences
include not having knowledge of the microlocal structure of $S$ near $\partial \mcB_Y$, and allowing for the possibility that the domain of the scattering map may
not be all of $\mcB$.

\begin{remark}\label{rmk:mapIsInjective}
Recall that for  $\overline{y}=(y,\eta)\in \mcB$ we write $\overline{y}'=(y,-\eta)$.   We shall use that since  $\kappa ( \kappa (\overline{y})')=\overline{y}'$, $\overline{y}\in \mathcal{D}_{\kappa}$ if and only if $\overline{y}'\in \mathcal{D}_{\kappa^{-1}}$, and similarly for iterates of $\kappa$.  Hence the condition we made
on the Minkowski content in assumption (\ref{as:Mink}) is equivalent to making the assumption for all $m\in \Z \setminus \{0\}$.  
\end{remark}
\begin{remark}
The examples described in Sections \ref{ss:singeex} and \ref{ss:wp} satisfy conditions (1) and (2).  We show in Section \ref{ss:smwp} that a surface of revolution with 
a bulge, as introduced in Section \ref{ss:bulge}, satisfies condition (3) as well.
\end{remark}

Let $\SU=\SU(h)= (I-h^2\DY)^{1/4}_+S(I-h^2\DY)^{-1/4}_+$ if $1\not \in \spec(h^2\DY)$, and $\SU(h)=\lim_{h'\uparrow h}\SU(h')$ if $1\in \spec(h^2\DY)$.  
It will be helpful to recall here that $S, \; \SU: \mchy \rightarrow \mchy,$ where $\mchy=\mathbbold{1}_{[0,1]}(h^2\DY)L^2(Y)$.
The operator $\SU$  is the 
unitary (on $\mch_Y$) scattering matrix.  We note that both the scattering matrix and the hypothesis \ref{as:fp} depend on the choice of coordinate $r$ on the the cylindrical end.

\begin{theorem} \label{thm:equidist} Suppose $(X,g)$ is an $n$-dimensional manifold with infinite cylindrical end, and
$(X,g)$ and the associated scattering map $\kappa$ satisfy all the conditions listed above.  Let $f\in C(\Sphere^1)$.  Then 
$$\lim_{h\downarrow 0} \left(h^{n-1}\tr_{\mchy}(f(\SU))\right) =\frac{c_{n-1} \vol(Y)}{2\pi}\int_0^{2\pi}f(e^{i\theta})d\theta$$
where $c_{n-1}$ is the usual Weyl constant in dimension $n-1$.
\end{theorem}

 Subscripts on the trace in this section  and the next indicate
the space in which the trace is taken.

An immediate corollary of this Theorem is the following equidistribution result.
\begin{corollary} Let $0\leq \theta_1< \theta_2< 2\pi$.  Then
$$\lim_{h\downarrow 0} \left(h^{n-1} N(\theta_1,\theta_2,h) \right)  =\frac{c_{n-1} \vol(Y)}{2\pi}(\theta_2-\theta_1)$$
where $N(\theta_1,\theta_2,h)$ is the number of eigenvalues of $\SU$ with argument between $\theta_1$ and $\theta_2$.
\end{corollary}

\subsection {Proof of Theorem \ref{thm:equidist}}\label{ss:edp}
We begin with a result on the structure of the iterates of the unitary scattering matrix.
\begin{lemma}\label{l:compFIO}  
Let $m\in \Z\setminus\{0\}$ and let $\psi_0 \in C_c^\infty(\mcB)$ be supported in  $\mathcal{D}_{\kappa^m}$.  Then under the hypotheses of Theorem \ref{thm:equidist},
$(\SU)^m \Oph(\psi_0)$ is a semiclassical Fourier integral operator associated to the graph of $\kappa^m$.  
\end{lemma}
\begin{proof}

Theorem \ref{thm:v1} or \ref{thm:v2} implies the result for $m=1$.

Now suppose the lemma has been proved for $1\leq m \leq m'$.  We shall show that it holds for $m=m'+1$, proving the lemma for positive $m$ by
induction.  Recall now we assume that $\supp \psi_0\subset \mathcal{D}_{\kappa^{m'+1}}$, and use $\mathcal{D}_{\kappa^{m'+1}}\subset \mathcal{D}_{\kappa^{m'}}$.  Choose 
$\psi_{m'}\in C_c^\infty(\mcB)$ to be supported on the domain of $\kappa$ and to be $1$ on 
$\{ \kappa^{m'}(y,\eta)\mid (y,\eta)\in \supp \psi_0\}$.  
Then choose $\psi_{sp,m'}\in C_c^\infty([0,1)) $ so that 
$(I-\psi_{sp,m'}(h^2\DY))\Oph(\psi_{m'})=O(h^\infty)$.
We write  
\begin{align*}& (\SU)^{m'+1} \Oph(\psi_0)\\ & = \SU ( \psi_{sp,m'}(h^2\DY)+I-\psi_{sp,m'}(h^2\DY))\Oph(\psi_{m'}) (\SU)^{m'} \Oph(\psi_0)\\ & 
\hspace{3mm}
+ \SU (I-\Oph(\psi_{m'}))(\SU)^{m'} \Oph(\psi_0) \\
& = 
\SU\psi_{sp,m'}(h^2\DY)\Oph(\psi_{m'})(\SU)^{m'} \Oph(\psi_0)  + \SU (I-\Oph(\psi_{m'}))(\SU)^{m'} \Oph(\psi_0) +O(h^\infty).
\end{align*}
That this is a semiclassical FIO associated to $\kappa^{m'+1}$ follows from
the inductive hypothesis, an application of Theorem \ref{thm:v1} or \ref{thm:v2}, and the composition properties of Fourier integral operators.  Thus 
concludes the proof for positive $m$.


We now turn to the result for $\SU^{-1}$.  
We shall use that since  $\kappa ( (\kappa (\overline{y}))')=\overline{y}'$, using the notation $(\mathcal{D}_{\kappa^{-1}})'=\{\overline{y} \mid \overline{y}'\in \mathcal{D}_{\kappa^{-1}}\}, $
 gives $(\mathcal{D}_{\kappa^{-1}})'= \mathcal{D}_\kappa$.

Lemma 3.1 of \cite{par} implies that $\SU^T=\SU$, where $\SU^T$ denotes the transpose of $\SU$.  Then for any 
${\psi}\in C_c^\infty(\mathcal{D}_\kappa)$,
$ \SU^T \Oph({\psi}) =\SU \Oph({\psi})$ is a semiclassical FIO associated to the scattering map $\kappa$.  Denote complex conjugation by 
$\mathcal{C}$, and let  $\psi_0\in C_c^\infty(\mathcal{D}_{\kappa^{-1}})$.  As an operator on $\mch_Y$,  $\SU^{-1}=\SU^*$
and $\SU^*\Oph(\psi_0)= \mathcal{C} \:\SU^T \mathcal{C}\Oph(\psi_0)$.  But $\mathcal{C}\Oph(\psi_0)=\Oph({\psi})\mathcal{C}$ for some ${\psi}\in C_c^\infty((\mathcal{D}_{\kappa^{-1}})')= C_c^\infty(\mathcal{D}_\kappa)$, so that
$\SU^*\Oph(\psi_0)= \mathcal{C} \: \SU \Oph({\psi})\mathcal{C}$.    Now using that we know that $\SU \Oph({\psi})$ is a semiclassical FIO, the properties of FIOs under
conjugation by the action of the complex conjugate $\mathcal{C}$,  and the equality of sets 
$\{ ((\kappa(\overline{y}))',\overline{y}') \mid \overline{y}\in \mathcal{D}_{\kappa}\}= \{ (\kappa^{-1}(\overline{y}),\overline{y})  \mid  \overline{y}\in \mathcal{D}_{\kappa^{-1}}\}$ we prove the 
second assertion in the special case $m=1$.

The general case of negative values of $m$ can be proved by induction, in much the same manner as for positive $m$.
\end{proof}

\begin{lemma}\label{l:psiepsilon}  Under the hypotheses of Theorem \ref{thm:equidist}, for any $m\in \Natural$, $\epsilon>0$ there is a $\psi\in C_c^\infty(\mathcal{D}_{\kappa^m}\cap \mathcal{D}_{\kappa^{-m}})$ so that for $h>0$ sufficiently small,
 $\| (I-\Oph(\psi))\mathbbold{1}_{[0,1]}(h^2\DY)\|_{\tr_{L^2(Y)}}\leq \epsilon h^{-n+1}$.
 \end{lemma}
 \begin{proof}
 For $m$ fixed and $\delta>0$, set 
 $$U_\delta:= \{ \overline{y}\in \mcB \mid \operatorname{dist}_{T^*Y}(\overline{y}, T^*Y \setminus \left( \mathcal{D}_{\kappa^m}\cap \mathcal{D}_{\kappa^{-m}})\right)>\delta\}$$
 and $V_\delta:= \mcB \setminus \overline{U_\delta}$.   Note that $U_\delta$ is open, and $U_\delta \subset U_{\delta/2}$.
 Let $\psi_\delta\in C_c^{\infty}(U_{\delta/2})\subset C_c^\infty(\mcB)$ satisfy $0\leq \psi_\delta\leq 1$ and $1-\psi_\delta=0$ on $U_{\delta}$.  
  
  Let $\chi_\delta\in C_c^\infty([0,1+\delta);[0,1])$ with $\chi_\delta(t)=1$ for $t\in [0,1]$, and note 
 \begin{align*}
 \| (I-\Oph(\psi_\delta))\mathbbold{1}_{[0,1]}(h^2\DY)\|_{\tr_{L^2(Y)}} & = \| (I-\Oph(\psi_\delta))\chi_\delta(h^2\DY))\mathbbold{1}_{[0,1]}(h^2\DY)\|_{\tr_{L^2(Y)}}\\
 & \leq  \| (I-\Oph(\psi_\delta))\chi_\delta(h^2\DY)\| _{HS_{L^2(Y)}} \| \mathbbold{1}_{[0,1]}(h^2\DY)\|_{HS_{L^2(Y)}}
 \end{align*}
 where $\|\bullet \|_{HS}$ denotes the Hilbert-Schmidt norm.  Now 
 \begin{align*}
& \| (I-\Oph(\psi_\delta))\chi_\delta(h^2\DY)\| _{HS_{L^2(Y)}}^2\\ & = \tr_{L^2(Y) }\left( \left( (I-\Oph(\psi_\delta))\chi_\delta(h^2\DY)\right) ^*(I-\Oph(\psi_\delta))\chi_\delta(h^2\DY) \right) \\
 & \leq C (2\pi h)^{1-n}\int_{T^*Y} \left| (1-\psi_\delta(y,\eta)) \chi_{\delta}(|\eta|)\right|^2d\mu+  O(h^{2-n})
 \end{align*}
 for some $C>0$ independent of $\delta$ and $h$.   Here $\mu$ is the Liouville measure.   By
 the Weyl law, $\| \mathbbold{1}_{[0,1]}(h^2\DY)\|_{HS_{L^2(Y)}}^2= c_{n-1}h^{1-n}\vol(Y)+O(h^{2-n})$.
 Thus there is a constant $C_0$ independent of $\delta$ and $h$ so that 
 \begin{equation}\label{eq:dint}
  \| (I-\Oph(\psi))\mathbbold{1}_{[0,1]}(h^2\DY)\|_{\tr_{L^2(Y)} }\leq C_0 h^{1-n} \left(\int_{T^*Y} \left| (1-\psi_\delta(y,\eta))\chi_{\delta}(|\eta|)\right|^2  \right)^{1/2} d\mu+ O(h^{2-n}).
  \end{equation}
 The integrand on the right in (\ref{eq:dint}) takes values in $[0,1]$ and is supported in $V_\delta\cup\{ \overline{y}\in T^*Y \mid 1\leq |\eta|\leq 1+\delta\}$.
 Let $W_\delta=\{\overline{y}=(y,\eta)\in T^*Y \mid  1-\delta< |\eta|<1+\delta\}$, and note 
 \begin{multline}\label{eq:VdWd}
 V_\delta \setminus \overline{W_\delta}\subset \left\{ \overline{y}\in \mcB \mid  \operatorname{dist}_{T^*Y}(\overline{y}, \mcB\setminus(\mathcal{D}_{\kappa^m}\cap \mathcal{D}_{\kappa^{-m}}))<\delta \right\}\\
 \subset \left\{ \overline{y}\in T^*Y\mid \operatorname{dist}_{T^*Y}((\overline{y}, \mcB \setminus(\mathcal{D}_{\kappa^m}\cap \mathcal{D}_{\kappa^{-m}}))<\delta \right\}.
 \end{multline}
 Since by hypothesis (\ref{as:Mink}) both $\mcB\setminus \mathcal{D}_{\kappa^m}$ and $\mcB\setminus \mathcal{D}_{\kappa^{-m}}$ have zero $(2n-2)$-dimensional Minkowski content, so does
 $\mcB\setminus (\mathcal{D}_{\kappa^m} \cap \mathcal{D}_{\kappa^{-m}})= (\mcB\setminus \mathcal{D}_{\kappa^m})\cup(\mcB\setminus \mathcal{D}_{\kappa^{-m}})$.
 Thus
 (\ref{eq:VdWd}) implies
 $\int_{V_\delta \setminus \overline{W_\delta}}1 d\mu \rightarrow 0$ as $\delta\downarrow 0$.  Of course $\int_{\overline{W_\delta }} 1d\mu \rightarrow 0$ as $\delta \downarrow 0$.  
 Hence,  since
 \begin{equation}
 \int_{T^*Y} \left| (1-\psi_\delta(y,\eta))\chi_{\delta}(|\eta|)\right|^2 d\mu   \leq \int_{V_\delta \cup W_\delta} 1 d\mu
 \end{equation}
 we may choose $\delta_0>0$ small enough so that 
 $$C_0 \left(\int_{T^*Y} \left|( (1-\psi_{\delta_0}(y,\eta))\chi_{\delta_0}(|\eta|)\right|^2  \right)^{1/2}d\mu <\epsilon/2.
 $$
   Then set $\psi=\psi_{\delta_0}$,  
 and we have chosen $\psi$ so that  
 $$\| (I-\Oph(\psi))\mathbbold{1}_{[0,1]}(h^2\DY)\|_{\tr_{L^2(Y)}}\leq (\epsilon/2) h^{-n+1}+O(h^{2-n}).$$ When $h>0$ is sufficiently small, we have the desired estimate.
 \end{proof}
 
 \begin{corollary}
\label{c:epsclose} Under the hypotheses of Theorem \ref{thm:equidist}, for any $m\in \Z$ and $\epsilon>0$ there is a $\psi\in C_c^\infty(\mathcal{D}_{\kappa^m}
\cap \mathcal{D}_{\kappa^{-m}})$ so that for $h$ sufficiently small
$|\tr_{\mch_Y}(f(\SU)(I-\Oph(\psi)) | \leq \epsilon h^{-n+1}\sup|f|$.
\end{corollary}
\begin{proof}
Let $\psi$ be as guaranteed by Lemma \ref{l:psiepsilon}.  Then
\begin{multline*}
|\tr_{\mch_Y}(f(\SU)(I-\Oph(\psi))|= |\tr_{L^2(Y)} (f(\SU)(I-\Oph(\psi)) \mathbbold{1}_{[0,1]}(h^2\DY)|  \\ \leq  \| f(\SU)\| \|(I-\Oph(\psi)) \mathbbold{1}_{[0,1]}(h^2\DY)\| _{\tr_{L^2(Y)}} \leq \sup|f| \epsilon h^{-n+1}.
\end{multline*}
\end{proof}

\begin{lemma}\label{l:trpowers} 
Let $m\in {\bf Z} \setminus\{ 0\}$, and $\psi\in C_c^\infty(\mathcal{D}_{\kappa^m})$.  Then under the hypotheses of Theorem \ref{thm:equidist}, $\tr_{\mch_Y} \left(\SU^m \Oph(\psi)\right) =o(h^{1-n})$.
\end{lemma}
\begin{proof}
 This follows from Lemma \ref{l:compFIO}, our hypothesis (\ref{as:fp}) on the fixed point set of $\kappa^m$, and \cite[Proposition 7.1]{GHZ}.
\end{proof}


\noindent {\em Proof of Theorem \ref{thm:equidist}}.
Given $\epsilon>0$ and $f\in C(\Sphere^1)$, use the 
density of the polynomials in $e^{i\theta}$ and $e^{-i\theta}$ in the continuous functions on $\Sphere^1$ to choose a  $q\in \C^\infty(\Sphere)$ with 
$q(e^{i\theta})=\sum_{j=-J}^Ja_j e^{ij\theta}$ for some $J\in \N$, $a_j\in \C$ and so
that $\sup |f(\theta) -q(\theta)|<\epsilon $.  Choose $\psi\in C_c^\infty(\mathcal{D}_{\kappa^J}\cap \mathcal{D}_{\kappa^{-J}})$ as guaranteed by Corollary \ref{c:epsclose}, applied with 
$m=J$.

Now 
\begin{equation}\label{eq:split}\tr_{\mch_Y}f(\SU)
= \tr _{\mch_Y}(f(\SU)- q(\SU))+\tr_{\mch_Y}(q(\SU)(I-\Oph(\psi))+ \tr_{\mch_Y}(q(\SU)\Oph(\psi)).
\end{equation}
Since by the Weyl law  $\mch_Y$ is  of dimension $h^{1-n}c_{n-1}\vol(Y) +O(h^{2-n})$ and $\|f(\SU)-q(\SU)\|<\epsilon$,
\begin{equation} \label{eq:est1}
|\tr_{\mch_Y} (f(\SU)- q(\SU))|< \epsilon h^{1-n}c_{n-1}\vol(Y) +O(h^{2-n}).
\end{equation}
By our choice of $\psi$ as in Corollary \ref{c:epsclose}, for $h>0$ sufficiently small
\begin{equation}\label{eq:est2}
|\tr_{\mch_Y}(q(\SU)(I-\Oph(\psi))| \leq \epsilon h^{1-n} \sup |q| \leq \epsilon h^{1-n} ( \epsilon+ \sup | f| ).
\end{equation}
Using $a_0=\frac{1}{2\pi}\int_0^{2\pi} q(e^{i\theta})d\theta$ and Lemma \ref{l:trpowers}, 
\begin{align} \nonumber \tr_{\mch_Y} (q(\SU)\Oph(\psi))& = \sum_{j=-J}^J a_j \tr_{\mch_Y}(\SU^j \Oph(\psi))\\
& =  \frac{1}{2\pi}\int_0^{2\pi} q(e^{i\theta})d\theta\; \tr_{\mch_Y} \Oph(\psi)+ o(h^{1-n}).
\end{align}
But by our choice of $\psi$ as in Corollary \ref{c:epsclose},  for $h$ sufficiently small 
\begin{equation*}
\left|\tr_{\mch_Y} \left(\Oph(\psi) -I_{\mchy}\right)\right|<\epsilon h^{1-n},
\end{equation*}
and since the dimension of $\mchy$ is $c_{n-1}\vol (Y) h^{1-n}+O(h^{2-n})$ by the Weyl law, 
\begin{equation} \label{eq:est3}
| \tr_{\mch_Y} \Oph(\psi) - c_{n-1}\vol (Y) h^{1-n}|<\epsilon h^{1-n}+ O(h^{2-n}).
\end{equation}
 Using (\ref{eq:est1}- \ref{eq:est3})  in (\ref{eq:split}), we find for $h$ sufficiently small
$$\left| \tr_{\mch_Y} f(\SU)- \frac{c_{n-1}}{2\pi}\vol (Y) h^{1-n} \int_0^{2\pi}f(e^{i\theta})d\theta \right | \leq 2 \epsilon h^{1-n} \left( c_{n-1}\vol(Y)  +\epsilon+ \sup | f| +1\right)   +o(h^{1-n})$$
implying 
$$\lim_{h\downarrow 0} \left|  h^{n-1}\tr_{\mch_Y}  f (\SU)- \frac{c_{n-1}}{2\pi}\vol (Y)  \int_0^{2\pi}f(e^{i\theta})d\theta \right |  \leq 2\epsilon \left( c_{n-1}\vol(Y)  +\epsilon+ \sup | f|  +1\right).$$
Since $\epsilon>0$ is arbitrary, this proves the theorem. \qed

\appendix

\section{Warped products with a bulge}\label{s:wpb}
This section collects two results for warped products with bulges, as introduced in Section \ref{ss:bulge}.  These results
are a resolvent estimate and a computation of the scattering map for the special case in which the
manifold is a surface of revolution.

We recall the setting.  Let $f\in C^\infty(\R;(0,\infty))$ satisfy $f(s)=1$ if $|s|>a$ and suppose $f$ has a
single nondegenerate critical point in $(-a,a)$, and this point is a maximum
of $f$.  Let $(Y_0,g_{Y_0})$ be a smooth compact Riemannian manifold, and set  $(X,g)=(\R \times Y_0, ds^2+ f^{4/(n-1)}g_{Y_0})$. 

\subsection{The Resolvent estimate for the warped product with a bulge}

Here we bound the microlocally cut-off resolvent on a warped product
with a bulge.
We give a result that is stronger than we need in terms of 
the spatial cut-off (a weight in $|s|$, rather than a compactly supported
function in $s$).  Our presentation uses a commutator
argument and is inspired by 
\cite{Vod, Dat} and references therein; see also \cite[Section 2]{CDI}.
\begin{lemma}\label{l:bulgeestimate}
Let $f$, $Y_0$, and $X$ be as described above.Then for any 
$\epsilon,\;\alpha>0$ there are $C_0=C_0(\epsilon,\alpha)$, $\;h_0=h_0(\epsilon,\alpha)>0$ so that
\begin{equation} \label{eq:specific}
 \|\mathbbold{1}_{[0,1-\epsilon]}(h^2\Delta_{Y_0})(1+|s|)^{-(1+\alpha)/2}(\Ph-1-i0)^{-1}
(1+|s|)^{-(1+\alpha)/2}\|\leq C_0h^{-1}\; \text{for $0<h\leq h_0$}.
\end{equation}
\end{lemma}
We emphasize that while $\epsilon>0$ is small, it is fixed here.

\begin{proof}

In this case 
\begin{equation}\label{eq:nicer}
h^2 \DX= f(s)^{-1} \left(-h^2\partial_s^2+h^2f''(s)/f(s) +h^2\Delta_{Y_0} f(s)^{-4/(n-1)}
\right) f(s).
\end{equation}
Since $f$ is bounded, and is bounded below away from $0$,
it suffices to study the resolvent of the operator in parentheses on the 
right hand side of (\ref{eq:nicer}).  To do so,
we will separate variables.  
Set $\varphi= f(s)^{-4/(n-1)}$.  We will 
show that for any $\alpha>0$ and $\epsilon>0$ there is a
$h_0=h_0(\epsilon,\alpha)$, $C_0=C_0(\epsilon,\alpha)$ so that 
\begin{multline}\label{eq:newgoal}
\|(1+|s|)^{-(1+\alpha)/2}(-h^2\partial_s^2 +\tau \varphi -1-i\delta)^{-1}(1+|s|)^{-(1+\alpha)/2}\|_{L^2(\R)\rightarrow L^2(\R)}\leq C_0h^{-1}\;\\ \text{for}\; 0<h<h_0,\; 0<\delta<1,\; 0\leq \tau \leq 1-\epsilon.
\end{multline}
Then using that this implies 
$\| h^2( f''/f) (-h^2\partial_s^2 +\tau \varphi -1-i\delta)^{-1}(1+|s|)^{-(1+\alpha)/2}\|=O(h)$, the estimate (\ref{eq:newgoal}) together with 
a separation of variables using (\ref{eq:nicer}) proves the lemma.

We give a proof of (\ref{eq:newgoal}) that is valid  uniformly
for all $\tau\in [0,1-\epsilon]$.  Without loss of generality we can assume that the 
maximum of $f$, and hence the minimum of $\varphi$, occurs at $s=0$
so that $s\varphi'(s)\geq 0$.  We 
also remark that $0<\varphi\leq 1$.  In order to simplify notation, 
we introduce  $Q_\tau:=-h^2 \partial_s^2+\tau \varphi-1$,
local to this proof.

Let $u\in H^2(\R)$ satisfy $u(s),\; u'(s)\rightarrow 0$ as $s\rightarrow \pm \infty$ 
and $(1+|s|)^{(1+\alpha)/2} (Q_\tau-i\delta)u\in L^2(\R)$.
Let $w\in C^1(\R;\R)$ be bounded, along with its first derivative.  
Now using inner products on $L^2(\R)$, add the equalities
\begin{align*}
\langle w' u, u\rangle & = -2\Re \langle w u, u'\rangle\\
h^2 \langle w' u', u' \rangle & = -2\Re \langle w h^2 u'', u'\rangle
\end{align*}
and 
$$ - \tau \langle (w\varphi)'u,u \rangle = 2\Re \langle \tau w \varphi u,u' \rangle$$
to get
\begin{equation}\label{eq:IBPidentity}
\langle w'u,u\rangle + h^2 \langle w' u',u'\rangle 
-\tau \langle (w\varphi)'u,u\rangle =
-2\Re \langle w (Q_\tau u -i\delta u), u'\rangle + 2\delta \Im 
\langle w u, u'\rangle .
\end{equation}
We wish to choose $w$ so that both $w'$ and $w'-\tau(w\varphi)'$ are nonnegative,
with $w'>0$.  To do so, set $w(s)=w_1(s)\varphi^\beta$, where $w_1(s)$ is 
the odd function that is given for $s>0$ by $w_1(s)=1-(1+s)^{-\alpha}$
and
$ \beta>0$ is  a constant to be chosen below.  The restriction
$\alpha, \beta>0$ ensures $w'>0$, since $w_1'>0$ and $w_1(s)\varphi'(s)\geq 0$.  
We compute
$$w'-\tau (w\varphi)'=
\varphi^{\beta-1}\left( w_1'\varphi(1-\tau\varphi)+w_1\varphi'(\beta(1-\tau \varphi)-\tau \varphi)\right).$$
Choosing $\beta= 2/\epsilon$ and using $\tau\leq 1-\epsilon$,
$0<\varphi\leq1$,  yields
\begin{align}\nonumber
w'-\tau (w\varphi)' &\geq \varphi^{\beta-1}\left( w_1'\varphi \epsilon  + 
w_1\varphi'( \beta \epsilon -1+\epsilon)\right)\\
 & \geq \varphi^{\beta} w_1' \epsilon.
\end{align}
Since $\epsilon>0$ is fixed
and  the minimum of $\varphi$ is strictly positive, there is a
$c_0>0$, independent of $\tau\in[0,1-\epsilon]$ so that 
\begin{equation}
w'-\tau (w\varphi)' \geq c_0 w_1' = c_0\alpha (1+|s|)^{-(1+\alpha)}.
\end{equation} 

Using these in (\ref{eq:IBPidentity}) and estimating the right hand 
side of  (\ref{eq:IBPidentity}) using the Cauchy-Schwarz inequality
yields, for some constant $C$ independent
of $h$, $\tau\in[0,1-\epsilon]$ and $\delta>0$, and any $\gamma >0$
\begin{equation}\label{eq:intermediate}
\| \sqrt{w_1'} u\|^2 + h^2 \| \sqrt{w'}u'\|^2 
\leq \frac{C}{\gamma h^2}\| (w/\sqrt{w'})(Q_\tau u -i\delta u)\| ^2 + C \gamma h^2 \|\sqrt{w'} u'\|^2 + C
\delta \| u\| \|u'\|.
\end{equation}
We will use below that we can simplify this somewhat, by using that
$w$ is bounded and that $w'\geq c_1 w_1'$ for some $c_1>0$.
Now 
\begin{multline}
\|u'\|^2 = \frac{1}{h^2} \langle -h^2 u'', u\rangle 
= \Re \frac{1}{h^2}\langle (Q_\tau-i\delta)u, u \rangle + \frac{1}{h^2} \langle
(1-\tau \varphi)u,u\rangle \\ \leq 
\frac{1}{h^2} \left \| \frac{1}{\sqrt{w_1'}}(Q_\tau-i\delta)u\right\| 
\left\|
 \sqrt{w_1'}u \right\| 
+\frac{1}{h^2} \left \|u\right\|^2
\end{multline}
and 
$$\delta \|u\|^2 =\Im \langle (Q_\tau -i \delta)u, u\rangle
\leq \left \| \frac{1}{\sqrt{w_1'}}(Q_\tau -i\delta)u\right\| \left\| \sqrt{w_1'} u \right\|
$$
giving, if $0<\delta\leq1$ and $\gamma>0$
\begin{equation}\label{eq:uu'est}
\delta \|u\| \| u'\| \leq \frac{1}{h} \left\| \frac{1}{\sqrt{w_1'}}(Q_\tau -i\delta)u\right\| \left\| \sqrt{w_1'} u \right\| (1+\delta)^{1/2}
\leq \frac{1}{\gamma h^2} \left \| \frac{1}{\sqrt{w_1'}}(Q_\tau -i\delta)u\right\|^2
+ \gamma \left\| \sqrt{w_1'} u \right\|^2.
\end{equation}
Using this in (\ref{eq:intermediate}) and simplifying as indicated 
above yields, for some constant 
$C$ independent of $\tau\in [0,1-\epsilon]$ and $\delta\in (0,1]$,
$$\| \sqrt{w_1'} u\|^2 + h^2 \| \sqrt{w'}u'\|^2 
\leq  \frac{C}{\gamma h^2} \left \| \frac{1}{\sqrt{w_1'}}(Q_\tau -i\delta)u\right\|^2 
+C \gamma h^2 \|\sqrt{w'} u'\|^2
+C \gamma \| \sqrt{w_1'} u \|^2.
$$
Choosing $\gamma $ sufficiently small, we can absorb the second and third
terms on the right into the corresponding terms on the left, 
yielding, on using estimates for $w'$, $w_1'$ and with a new constant $C$
$$\| (1+|s|)^{-(1+\alpha)/2} u\|^2 + h^2 \| (1+|s|)^{-(1+\alpha)/2}u'\|^2 
\leq  \frac{C}{ h^2} \left \|(1+|s|)^{(1+\alpha)/2}(Q_\tau -i\delta)u\right\|^2 .
$$
Dropping the second term on the left and applying the 
resulting inequality with $u=(Q_\tau-i\delta)^{-1}(1+|s|)^{-(1+\alpha)/2}v$ for
$v\in L^2(\R)$
proves
(\ref{eq:newgoal}).
\end{proof}

We remark that the
estimate (\ref{eq:newgoal}) holds for any fixed $\tau \in [0,1-\epsilon]$
from well-known non-trapping results, 
 e.g. \cite{RT,GM,Bur}.  In fact, a rescaling and these known non-trapping
results prove the estimate uniformly
for 
$\tau\in [\epsilon', 1-\epsilon]$
for any fixed $\epsilon,\; \epsilon'>0$.
We are unaware, however, of a result that directly implies (\ref{eq:newgoal})
uniformly for all $\tau \in [0,1-\epsilon]$, so we have chosen to give a direct proof
here, valid for all values of $\tau$ in this interval.

\subsection{Scattering map for a surface of revolution with a bulge} \label{ss:smwp}

In this section we compute the scattering map for a surface of revolution with a bulge, as in Section \ref{ss:bulge}.   We use the function $f$ and manifold $X$ introduced above
(Section \ref{s:wpb} or \ref{ss:bulge}), but 
specialize to the case $Y_0=\Sphere^1$ and $X=\R \times \Sphere^1$.


It will be convenient to use a coordinate $\theta\in \R$ on $\Sphere^1$, identifying points which differ by an integral multiple of $2\pi$.   The manifold  $X$ has two connected ends and $Y=\Sphere^1_L\sqcup \Sphere^1_R$, where $\Sphere^1_L$ corresponds to 
$s\rightarrow -\infty$, the ``left" end. On $\Sphere^1_L$ and $\Sphere_R^1$ we use the coordinate $\theta$ which is inherited from the factor of $\Sphere^1$ in $X$.  

\begin{lemma}\label{l:kappabulge} Let $X=\R_s\times \Sphere^1_\theta$ be a surface of revolution with a bulge as described above, and let $\{r=0\}=\{s=-a-4\}\sqcup\{s=a+4\}$.  Then if  $ (\theta_-,\eta_-)\in T^* \Sphere^1_R$ with $|\eta_-|<1$,
$$ \kappa(\theta_-,\eta_-)=\left(\theta_-+\eta_-\int_{-a-4}^{a+4}\frac{1}{f^2(\ts)\sqrt{(f(\ts))^4-\eta_-^2}}
d\ts,\eta_-\right)\in T^*\Sphere^1_L.$$
On the other hand, if  
$ (\theta_-,\eta_-)\in T^* \Sphere^1_L$ and $|\eta_-|<1$,  then $\kappa(\theta_-,\eta_-)\in T^*\Sphere^1_R$
and it is given by the same expression.
\end{lemma}


\begin{proof}
We use the coordinates $(s,\theta,\rho,\eta)$ on $T^*X$.  (We are going back to the notation used at 
the beginning of the paper where the spatial variables come first, followed by the corresponding fiber variables.)
The principal symbol of the Laplacian is $p=\rho^2+(f(s))^{-4} \eta^2$.  Thus the equations for the Hamiltonian flow are 
\begin{equation}\label{eq:hameq}\begin{array}{ll}
\dot{s}=2\rho  \hspace{10mm}& \dot{\theta}=2(f(s))^{-4}  \eta\\
\dot{\rho}= 4(f(s))^{-5} f'(s)  \eta^2 & \dot{\eta}=0.
\end{array}
\end{equation}
Denote the initial conditions by $(s_0,\theta_0,\rho_0,\eta_0)$, and note that  $\eta$ is constant under the Hamiltonian flow, while $s$ and $\rho$ are 
independent of $\theta_0$.  Thus, denoting the Hamiltonian flow by $\Phi_t$, we have
$$\Phi_t(s_0,\theta_0,\rho_0,\eta_0)=(s(t, s_0,\rho_0,\eta_0), \theta(t,s_0,\theta_0, \rho_0,\eta_0),\rho(t,s_0, \rho_0,\eta_0), \eta_0).$$

We shall prove the first equality of the lemma; the second can be derived from the first.  Thus we wish to consider initial data
\begin{equation}
\label{eq:initialdata}
(s_0,\theta_0,\rho_0,\eta_0)=(a+4, \theta_0, -\sqrt{1-\eta_0^2},\eta_0),\; \text{where $|\eta_0|<1$}.
\end{equation}
Since $p$ is constant under the Hamilton flow, $\rho^2 +f^{-4}(s) \eta^2= \rho_0^2 + \eta_0^2= 1$ using that the initial data are as in (\ref{eq:initialdata}).  Thus
since $f(s)\geq 1$ and $\rho_0<0$, $\rho = -( 1-f^{-4}(s)\eta^2)^{1/2}$.  Using (\ref{eq:hameq}) shows that $s$ is a strictly decreasing function of $t$ for such initial data.  Thus
$\kappa(\theta_0,\eta_0)\in T^*\Sphere^1_L$, 
and we wish to find $\theta( t_{-a-4},s_0,\theta_0, \rho_0,\eta_0)$ where $t_{-a-4}$ is the value of $t$ for which $s(t, s_0,\rho_0,\eta_0)=-a-4$.  This value of $t$ depends on 
$\eta_0$, but we suppress this in our notation.  Using (\ref{eq:hameq}), 
\begin{equation}\label{eq:thetadiff}
\theta( t_{-a-4},s_0,\theta_0, \rho_0,\eta_0)-\theta_0=\eta_0\int_0^{t_{-a-4}}2 f^{-4}(s(t,s_0,\rho_0, \eta_0))dt.
\end{equation}

To evaluate the integral in (\ref{eq:thetadiff}) we shall think of $s$, rather than $t$, as the independent variable, which works since
$s$ is a strictly decreasing function of $t$.  Using (\ref{eq:hameq}) to 
find the derivative of $t$ with respect to $s$ gives
$$\theta( t_{-a-4},s_0,\theta_0, \rho_0,\eta_0)-\theta_0 =2\eta_0\int_{a+4}^{-a-4} f^{-4}(\ts)\frac{1}{2\rho (\ts)} d\ts = 
\eta_0 \int^{a+4}_{-a-4} f^{-4}(\ts)\frac{1}{\sqrt{1-f^{-4}(\ts) \eta^2_0}} d\ts $$
where we use $\ts$ as a variable to emphasize it is not a function of $t$ here.  

\end{proof}
A similar, but more complicated, computation can be made for an hourglass-type surface of revolution.

Using Lemma \ref{l:kappabulge} we can see that for a surface of revolution with a bulge the scattering map $\kappa$ satisfies Hypothesis 3 of Section \ref{s:eps}.  Indeed, it
is clear that for any $m\in \Z$, $\kappa^{2m+1}$ has no fixed points.   Moreover, for each fixed value of $\eta_-$, the $\theta$ component of 
$\kappa^{2m}(\bullet,\eta_-)$ is a rotation
by $m \delta_\theta(\eta_-)$, where 
$\delta_\theta(\eta_-):= 2\eta_-\int_{-a-4}^{a+4}\frac{1}{f^2(\ts)\sqrt{(f(\ts))^4-\eta_-^2}}d\ts.$  Thus fixed points of $\kappa^{2m}$
correspond to values of $\eta_-$ so that
$m \delta_\theta(\eta_-)$ is an integral multiple of $2\pi$.  But since $\delta_\theta$ is a smooth, strictly increasing function of $\eta_-\in (-1,1)$, 
for $m\not = 0$ this can happen only for 
isolated vales of $\eta_-$, with accumulation points only at $\eta_-=\pm 1$.

\end{document}